\documentclass[12pt]{article}

\usepackage{a4,amsmath,amsthm,amsfonts,latexsym,amssymb,xspace,graphicx,hyperref,mathrsfs,epsfig,natbib,bm,bbm,enumitem}
\usepackage[retainorgcmds]{IEEEtrantools}
\usepackage{todonotes}

\theoremstyle{definition}
\newtheorem{theorem}{Theorem}[section]
\newtheorem{proposition}[theorem]{Proposition}
\newtheorem{lemma}[theorem]{Lemma}
\newtheorem{corollary}[theorem]{Corollary}
\newtheorem{definition}{Definition}[section]
\newtheorem{remark}{Remark}[section]

\newtheorem{example}{Example}[section]
\numberwithin{equation}{section}

\newcommand{\cip}{\,\perp\!\!\!\perp}

\newcommand{\as}[1]{\mbox{\rm a.s.\ [$#1$]}}
\newcommand{\cS}{{\cal S}}

\newcommand{\cA}{{\cal A}}
\newcommand{\cG}{{\cal G}}
\newcommand{\cF}{{\cal F}}

\newcommand{\cB}{{\cal B}}
\newcommand{\cD}{{\cal D}}

\newcommand{\cP}{{\cal P}}
\newcommand{\mE}{\mathbb{E}}
\newcommand{\mP}{\mathbb{P}}
\newcommand{\mN}{\mathbb{N}}
\newcommand{\mR}{{\mathbb R}}
\newcommand{\indo}[2]{\mbox{$#1 \,\cip\, #2$}}
\newcommand{\ind}[3]{\mbox{$#1 \, \cip\, #2 \mid #3$}}
\newcommand{\sind}[3]{\mbox{$#1 \, \cip_{s}\, #2 \mid #3$}}
\newcommand{\vind}[3]{\mbox{$#1 \, \cip_{v}\, #2 \mid #3$}}
\newcommand{\pind}[3]{\mbox{$#1 \, \cip_{p}\, #2 \mid #3$}}

\renewcommand{\b}[1]{\boldsymbol{#1}}
\newcommand{\bX}{\b{X}}

\newcommand{\ie}{{\em i.e.\/}\xspace}
\newcommand{\eg}{{\em e.g.\/}\xspace}
\newcommand{\etc}{{\em etc.\/}\xspace}

\newcommand{\Ci}{Conditional independence}
\newcommand{\ci}{conditional independence}
\newcommand{\cis}{conditional independence}

\newcommand{\eci}{extended conditional independence}
\newcommand{\ecis}{extended conditional independence}

\newcommand{\pecis}{pairwise extended conditional independence}

\newcommand{\sci}{stochastic conditional independence}
\newcommand{\scis}{stochastic conditional independence}
\newcommand{\Vci}{Variation conditional independence}

\newcommand{\vcis}{variation conditional independence}

\newcommand{\dts}{Decision Theoretic framework}

\renewcommand{\eqref}[1]{\mbox{(\ref{eq:#1})}}
\newcommand{\secref}[1]{\mbox{Section~\ref{sec:#1}}}

\newcommand{\lemref}[1]{\mbox{Lemma~\ref{lem:#1}}}

\newcommand{\defref}[1]{\mbox{Definition~\ref{def:#1}}}

\newcommand{\itref}[1]{\mbox{\ref{it:#1}}}
\newcommand{\thmref}[1]{\mbox{Theorem~\ref{thm:#1}}}
\newcommand{\corref}[1]{\mbox{Corollary~\ref{cor:#1}}}
\newcommand{\proporef}[1]{Proposition~\ref{propo:#1}}

\newcommand{\appref}[1]{\mbox{Appendix~\ref{app:#1}}}
\newcommand{\remref}[1]{Remark~\ref{rem:#1}}
\newcommand{\ace}{\mbox{\rm ACE}}

\newcommand{\halm}{\hspace*{\fill} $\Box$\par}
\newenvironment{proof0}[1]{\noindent {\bf Proof of
    {#1}. }}{\halm\vspace{\baselineskip}}
\newenvironment{ex}{\begin{example}\rm}{\halm\end{example}}

\title{Extended Conditional Independence and Applications in Causal
  Inference} \author{Panayiota Constantinou and A. Philip Dawid}

\begin{document}
\maketitle

\begin{abstract}
  The goal of this paper is to integrate the notions of \scis\ and
  \vcis\ under a more general notion of \eci.  We show that under
  appropriate assumptions the calculus that applies for the two cases
  separately (axioms of a separoid) still applies for the extended
  case.  These results provide a rigorous basis for a wide range of
  statistical concepts, including ancillarity and sufficiency, and, in
  particular, the \dts\ for statistical causality, which uses the
  language and calculus of conditional independence in order to
  express causal properties and make causal inferences.

\end{abstract}

\section{Introduction}
\label{sec:1}

\Ci\ is a concept that has been widely studied and used in Probability
and Statistics.  The idea of treating \cis\ as an abstract concept
with its own calculus was introduced by~\citet{daw:79a}, who showed
that many results and theorems concerning statistical concepts such as
ancillarity, sufficiency, causality \etc, are just applications of
general properties of \cis---extended to encompass stochastic and
non-stochastic variables together.  Properties of \cis\ have also been
investigated by~\citet{spo:94} in connection with causality, and
\citet{peapaz:86, pea:00,geiver:90,laudaw:90} in connection with
graphical models.

In this paper, we consider two separate concepts of \ci: \emph{\scis},
which involves solely stochastic variables, and \emph{\vcis}, which
involves solely non-stochastic variables.  We argue that, although
these concepts are fundamentally different in terms of their
mathematical definitions, they share a common intuitive understanding
as ``irrelevance'' relations.  This allows them to satisfy the same
set of rules (axioms of a separoid~\citep{daw:01a}).  Armed with this
insight, we unify the two notions into the more general concept of
\emph{\ecis}, and show that (under suitable technical conditions)
\ecis\ also satisfies the axioms of a separoid.

To motivate the need for such a theory we recall some fundamental
concepts of statistics.  First, consider the concept of
\emph{ancillarity}~\citep{fis:25}.  Let $\b{X}:=(X_1,X_2,\ldots,X_n)$
be a random sample from a probability distribution with unknown
parameter $\theta$, and let $T=T(\b{X})$ be a statistic.  $T$ is
called an \emph{ancillary statistic for $\theta$} if its distribution
does not depend on the value of $\theta$~\citep{bas:64}.  For example,
consider an independent and identically distributed sample
$(X_1,X_2,\ldots,X_n)$ from the normal $\mathcal{N}(\theta, 1)$
distribution.  Then the range
$T:= \max\{X_1,X_2,\ldots,X_n\} -\min\{X_1,X_2,\ldots,X_n\}$ is an
ancillary statistic, because its distribution does not change as
$\theta$ changes.  Ancillary statistics can be used to recover
information lost by reducing the data to the maximum likelihood
estimate~\citep{ghorei:10}.  For our purposes, we remark that the
definition of ancillarity can be understood intuitively as requiring
the independence of the stochastic variable $T$ from the
non-stochastic variable $\theta$.  We can express this property using
the now standard (conditional) independence notation of
\citet{daw:79a}: $\indo{T}{\theta}$.

Another example is the notion of \emph{sufficiency}~\citep{fis:25}.
With notation as above, $T$ is a \emph{sufficient statistic for
  $\theta$} if the conditional distribution of the full data $\b{X}$,
given the value of $T(\b{X})$, does not depend on the value of the
parameter $\theta$~\citep[p.~272]{casber:01}.  For example, consider
an independent and identically distributed sample
$\bX = (X_1,X_2,\ldots,X_n)$ from the Poisson distribution with mean
$\theta$.  Then the sample total $T= X_1+X_2+ \ldots+X_n$ is a
sufficient statistic for $\theta$, since the distribution of $\bX$,
given $T=t$, is multinomial ${\cal M}(t; 1/n,\ldots,1/n)$ for all
$\theta$.  Here we emphasize that sufficiency can be expressed
intuitively as: ``Given $T$, $\bX$ is independent of $\theta$'', where
$\b{X}$ and $T$ are stochastic variables and $\theta$ is a
non-stochastic variable.  Using conditional independence notation:
\ind{\b{X}}{\theta}{T}.

A further application of these ideas emerges from the area of
causality: in particular, the \dts\ of statistical causality
\citep{apd:annrev}.  In this framework, the language and calculus of
conditional independence are fundamental for expressing and
manipulating causal concepts.  The \dts\ differentiates between
observational and interventional regimes, using a non-stochastic
variable to index the regimes.  Typically, we consider the regime
under which data can be collected (the \emph{observational} regime)
and a number of \emph{interventional} regimes that we wish to compare.
Since we mostly have access to purely observational data, we focus on
extracting information from the observational regime relevant to the
interventional regimes.  Then the conditions that would justify
transfer of information across regimes can be expressed in the
language of conditional independence.  To illustrate this, suppose we
are interested in assessing the effect of a binary treatment variable
$T$ on a disease outcome variable $Y$ (\eg, recovery).  Denote
\begin{equation*}
  T = \left\{
    \begin{array}{rl}
      0, & \text{for control treatment} \\
      1, & \text{for active treatment.}
    \end{array} \right.
\end{equation*}
We consider three regimes, indexed by a non-stochastic variable
$\Sigma$:
\begin{equation*}
  \Sigma = \left\{
    \begin{array}{rl}
      \emptyset & \text{denotes the observational regime} \\
      0 & \text{denotes the interventional regime under control treatment} \\
      1 & \text{denotes the interventional regime under active treatment.}
    \end{array} \right.
\end{equation*}
In the simplest case, we might entertain the following (typically
unrealistic!) property: for either treatment choice $T=0,1$, the
conditional distribution of the disease variable $Y$, given the
treatment variable $T$, is the same in the observational and the
corresponding interventional regime.  We can express this property,
using conditional independence notation, as \ind{Y}{\Sigma}{T}.  Such
a property, when it can be taken as valid, would allow us to use the
observational regime to make causal inference directly.  However, in
most cases this assumption will not be easy to defend.  Consequently
we would like to explore alternative, more justifiable, conditions,
which would allow us to make causal inference.  For such exploration a
calculus of \ecis\ becomes a necessity.

The layout of the paper is as follows.  In \secref{2} we give the
definition of a \emph{separoid}, an algebraic structure with five
axioms, and show that \scis\ and \vcis\ satisfy these axioms.  In
\secref{3} we rigorously define \emph{\ecis}, a combined form of
stochastic and variation conditional independence, and explore
conditions under which \ecis\ satisfies the separoid axioms, for the
most part restricting to cases where the left-most term in an \ecis\
relation is purely stochastic.  In \secref{bayes} we take a Bayesian
approach, which allows us to deduce the axioms when the regime space
is discrete.  Next, using a more direct measure-theoretic approach, we
show in \secref{nodiscrete} that the axioms hold when all the
stochastic variables are discrete, and likewise in the presence of a
dominating regime.  In \secref{pairwise} we introduce a slight
weakening of \ecis, for which the axioms apply without further
conditions.  Next, \secref{tcodrv} attempts to extend the analysis to
cases where non-stochastic variables appear in the left-most term.
Our analysis is put to use in \secref{4}, which gives some examples of
its applications in causal inference, illustrating how \eci, equipped
with its separoid calculus, provides a powerful tool in the area.  We
conclude in \secref{disc} with some comments on the usefulness of
combining the theory of \ecis\ with the technology of graphical
models.

\section{Separoids}
\label{sec:2}

In this Section we describe the algebraic structure called a
\emph{separoid}~\citep{daw:01a}: a three-place relation on a join
semilattice, subject to five axioms.

Let $V$ be a set with elements denoted by $x,y,\dots$, and $\leq$ a
quasiorder (a reflexive and transitive binary relation) on $V$.  For
$x,y \in V$, if $x\leq y$ and $y \leq x$, we say that $x$ and $y$ are
\emph{equivalent} and write $x \approx y$.  For a subset
$A \subseteq V$, $z$ is a \emph{join} of $A$ if $a \leq z$ for all
$a\in A$, and it is a minimal element of $V$ with that property; we
write $z = \bigvee A$; similarly, $z$ is a \emph{meet} of $A$
($z = \bigwedge A$) if $z\leq a$ for all $a\in A$, and it is a maximal
element of $V$ with that property.  We write $x\vee y$ for
$\bigvee \{x,y\}$, and $x\wedge y$ for $\bigwedge \{x,y\}$.

Clearly if $z$ and $w$ are both joins (respectively meets) of $A$,
then $z \approx w$.  We call $(V,\leq)$ (or, when $\leq$ is
understood, just $V$) a \emph{join semilattice} if there exists a join
for any nonempty finite subset; similarly, $(V,\leq)$ is a \emph{meet
  semilattice} if there exists a meet for any nonempty finite subset.
When $(V,\leq)$ is both a meet and join semilattice, it is a
\emph{lattice}.
 
\begin{definition}[Separoid]
  \label{def:separoid}
  Given a ternary relation $\cdot \cip \cdot \mid \cdot$ on $V$, we
  call $\cip$ a \emph{separoid} $($on $(V,\leq))$, or the triple
  $(V,\leq,\cip)$ a \emph{separoid}, if:
  \begin{itemize}
  \item[S1:] $(V,\leq)$ is a join semilattice
  \end{itemize}
  and
  \begin{itemize}
  \item[P1:] \ind{x}{y}{z} $\Rightarrow$ \ind{y}{x}{z}
  \item[P2:] \ind{x}{y}{y}
  \item[P3:] \ind{x}{y}{z} and $w\leq y$ $\Rightarrow$ \ind{x}{w}{z}
  \item[P4:] \ind{x}{y}{z} and $w\leq y$ $\Rightarrow$
    \ind{x}{y}{(z\vee w)}
  \item[P5:] \ind{x}{y}{z} and \ind{x}{w}{(y\vee z)} $\Rightarrow$
    \ind{x}{(y \vee w)}{z}
  \end{itemize}
\end{definition}

The following Lemma shows that, in P4 and P5, the choice of join does
not change the property.

\begin{lemma}
  Let $(V,\leq,\cip)$ be a separoid and $x_i,y_i,z_i \in V$ ($i=1,2$)
  with $x_1\approx x_2$, $y_1\approx y_2$ and $z_1\approx z_2$.  If
  \ind{x_1}{y_1}{z_1} then \ind{x_2}{y_2}{z_2}.
\end{lemma}

\begin{proof}
  See Corollary~1.2 in~\citet{daw:01a}.
\end{proof}

\begin{definition}[Strong separoid]
  We say that the triple $(V,\leq,\cip)$ is a \emph{strong separoid}
  if we strengthen S1 in \defref{separoid} to
  \begin{itemize}
  \item[S$1'$:] $(V,\leq)$ is a lattice
  \end{itemize}
  and in addition to P1--P5 we require
  \begin{itemize}
  \item[P6:] if $z\leq y$ and $w\leq y$, then \ind{x}{y}{z} and
    \ind{x}{y}{w} $\Rightarrow$ \ind{x}{y}{(z\wedge w)}.
  \end{itemize}
\end{definition}

\subsection{Stochastic conditional independence as a separoid}

The concept of stochastic conditional independence is a familiar
example of a separoid (though not, without further conditions, a
strong separoid~\citep{daw:79b}).

Let $(\Omega,\cA)$, $(F,\cF)$ be measure spaces, and
$Y:\Omega\rightarrow F$ a random variable.  We denote by $\sigma(Y)$
the $\sigma$-algebra generated by $Y$, \ie\
$\{Y^{-1}(C): C \in \cF\}$.  We write
$Y:(\Omega,\cA) \rightarrow (F,\cF)$ to imply that $Y$ is measurable
with respect to the $\sigma$-algebras $\cA$ and $\cF$; equivalently,
$\sigma(Y)\subseteq \cA$.

\begin{lemma}
  \label{lem:measurability}
  Let $Y: (\Omega,\sigma(Y)) \rightarrow (F_{Y},\cF_{Y})$ and
  $Z: (\Omega,\sigma(Z)) \rightarrow (F_{Z}, \cF_{Z})$ be surjective
  random variables.  Suppose that $\cF_Y$ contains all singleton sets
  $\{y\}$.  Then the following are equivalent:
  \begin{enumerate}[label=(\roman{enumi})]
  \item
    \label{it:aa1}
    $\sigma(Y) \subseteq \sigma(Z)$.  \
  \item
    \label{it:aa2}
    There exists measurable
    $f: (F_{Z},\cF_{Z}) \rightarrow (F_{Y},\cF_{Y})$ such that
    $Y=f(Z)$.
  \end{enumerate}
\end{lemma}

\begin{proof}
  See \appref{A}.
\end{proof}

In the sequel, whenever we invoke \lemref{measurability} we shall
implicitly assume that its conditions are satisfied.  In most of our
applications of \lemref{measurability}, both $(F_Y, \cF_Y)$ and
$(F_Z, \cF_Z)$ will be the real or extended real line equipped with
its Borel $\sigma$-algebra $\cB$.

We recall Kolmogorov's definition of conditional
expectation~\citep[p.~445]{bil:95}:
\begin{definition}[Conditional Expectation]
  Let $X$ be an integrable real-valued random variable defined on a
  probability space $(\Omega,\cA,\mP)$ and let $\cG \subseteq \cA$ be
  a $\sigma$-algebra.  A random variable $Y$ is called (a version of)
  the \emph{conditional expectation of $X$ given $\cG$}, and we write
  $Y=\mE(X \mid \cG)$, if
  \begin{enumerate}[label=(\roman{enumi})]
  \item
    \label{it:ce1}
    $Y$ is $\cG$-measurable; and
  \item\label{it:ce2} $Y$ is integrable and
    $\mE(X\mathbbm{1}_A)=\mE(Y\mathbbm{1}_A)$ for all $A \in \cG$.
  \end{enumerate}
\end{definition}
When $\cG=\sigma(Z)$ we may write $\mE(X \mid Z)$ for
$\mE(X \mid \cG)$.

It can be shown that $\mE(X \mid \cG)$ exists, and any two versions of
it are almost surely equal.  In particular, if $X$ is $\cG$-measurable
then $\mE(X \mid \cG)= X$ \rm{a.s.}\@ Thus for any integrable function
$f(X)$, $\mE\{f(X) \mid X\}= f(X)$ \rm{a.s.}\@ Also by \itref{ce2} for
$A=\Omega$, $\mE(X)=\mE\{\mE(X \mid \cG)\}$.  We will use this in the
form $\mE(X)= \mE\{\mE(X \mid Y)\}$.

\begin{definition}[Conditional Independence]
  \label{def:scidef}
  Let $X,Y,Z$ be random variables on $(\Omega,\cA,\mP)$.  We say that
  \emph{$X$ is (conditionally) independent of $Y$ given $Z$ (with
    respect to $\mP$)}, and write $\sind{X}{Y}{Z}\,\, [\mP]$, or just
  $\sind{X}{Y}{Z}$ when $\mP$ is understood, if:
  $$\mbox{For all $A_{X} \in
    \sigma(X)$}, \,\,\mE\left(\mathbbm{1}_{A_{X}} \mid Y,Z\right)=
  \mE\left(\mathbbm{1}_{A_{X}} \mid Z\right)\,\, \as{\mP}.$$
\end{definition}

We refer to the above property as \emph{stochastic conditional
  independence}; we use the subscript $s$ under $\cip$ ($\cip_{s}$) to
emphasize that we refer to this stochastic definition.

To prove the axioms, we need equivalent forms of the above definition.

\begin{proposition}
  \label{propo:stocci}
  Let $X,Y,Z$ be random variables on $(\Omega,\cA,\mP)$.  Then the
  following are equivalent.
  \begin{enumerate}[label=(\roman{enumi})]
  \item
    \label{it:scidef1}
    \label{it:ascidef2}
    \sind{X}{Y}{Z}.
  \item
    \label{it:scidef3}
    \label{it:ascidef3}
    For all real, bounded and measurable functions $f(X)$,
    \mbox{$\mE\{f(X) \mid Y,Z\}$}$=\mE\{f(X) \mid Z\}$ \rm{a.s.}
  \item
    \label{it:scidef4}
    \label{it:ascidef4}
    For all real, bounded and measurable functions $f(X), g(Y)$,
    \mbox{$\mE\{f(X)g(Y) \mid
      Z\}$}$=\mE\{f(X) \mid Z\}\mE\{g(Y)\mid Z\}$ \rm{a.s.}
  \item\label{it:scidef2}
    \label{it:ascidef1}
    For all $A_{X} \in \sigma(X)$ and all $A_{Y} \in \sigma(Y)$,
    \begin{math}
      \mE(\mathbbm{1}_{A_{X}\cap A_{Y}} \mid Z)=
      \mE(\mathbbm{1}_{A_{X}} \mid Z)\mE(\mathbbm{1}_{A_{Y}} \mid Z)
      \,\, \text{\rm{a.s.}}
    \end{math}
  \end{enumerate}
\end{proposition}

\begin{proof}
  See \appref{A}.
\end{proof}

Henceforth we write $X \preceq Y$ when $X =f(Y)$ for some measurable
function $f$.

\begin{theorem}{(Axioms of Conditional Independence.)}
  \label{thm:saxioms}
  Let $X,Y,Z,W$ be random variables on $(\Omega,\cA,\mP)$.  Then the
  following properties hold (the descriptive terms are those assigned
  by~\citet{pea:88}):
  \begin{itemize}
  \item[P$1^{s}$.][Symmetry] \sind{X}{Y}{Z} $\Rightarrow$
    \sind{Y}{X}{Z}.
  \item[P$2^{s}$.] \sind{X}{Y}{Y}.
  \item[P$3^{s}$.][Decomposition] \sind{X}{Y}{Z} and $W \preceq Y$
    $\Rightarrow$ \sind{X}{W}{Z}.
  \item[P$4^{s}$.][Weak Union] \sind{X}{Y}{Z} and $W \preceq Y$
    $\Rightarrow$ \sind{X}{Y}{(W,Z)}.
  \item[P$5^{s}$.][Contraction] \sind{X}{Y}{Z} and \sind{X}{W}{(Y,Z)}
    $\Rightarrow$ \sind{X}{(Y,W)}{Z}.
  \end{itemize}
\end{theorem}

\begin{proof}\quad\\
  P$1^{s}$.  Follows directly from \proporef{stocci}\ \itref{scidef2}.
  
  \noindent P$2^{s}$.  Let $f(X),g(Y)$ be real, bounded and measurable
  functions.  Then
  \begin{align*}
    \mE\left\{f(X)g(Y) \mid Y\right\}&= g(Y)\, \mE\left\{f(X) \mid Y\right\} \,\, \rm{a.s.} \\
                                     &= \mE\left\{f(X) \mid Y\right\} \mE\left\{g(Y) \mid Y\right\}
                                       \,\, \rm{a.s.}
  \end{align*}
  which proves P$2^{s}$.
  
  \noindent P$3^{s}$.  Let $f(X)$ be a real, bounded and measurable
  function.  Since $W \preceq Y$, it follows from
  \lemref{measurability} that $\sigma(W)\subseteq \sigma(Y)$ and thus
  $\sigma(W,Z)\subseteq \sigma(Y,Z)$.  Then
  \begin{align*}
    \mE\left\{f(X) \mid W,Z\right\}&= \mE\left[\mE\left\{f(X) \mid Y,Z\right\} \mid W,Z\right] \,\, \rm{a.s.} \\
                                   &= \mE\left[\mE\left\{f(X) \mid Z\right\} \mid W,Z\right] \,\, \text{\rm{a.s.} since \sind{X}{Y}{Z}} \\
                                   &= \mE\left\{f(X) \mid Z\right\} \,\, \rm{a.s.}
  \end{align*}
  which proves P$3^{s}$.
  
  \noindent P$4^{s}$.  Let $f(X)$ be a real, bounded and measurable
  function.  Since $W \preceq Y$, it follows from
  \lemref{measurability} that $\sigma(W)\subseteq \sigma(Y)$ and thus
  $\sigma(Y,W,Z)=\sigma(Y,Z)$.  Then
  \begin{align*}
    \mE\left\{f(X)\mid Y,W,Z\right\} &= \mE\left\{f(X)\mid Y,Z\right\}\,\,\text{\rm{a.s.}} \\
                                     &= \mE\left\{f(X) \mid Z\right\} \,\, \text{\rm{a.s.} since \sind{X}{Y}{Z}} \\
                                     &= \mE\left\{f(X) \mid W,Z\right\} \,\, \text{\rm{a.s.} by
                                       P$3^{s}$}
  \end{align*}
  which proves P$4^{s}$.
  
  \noindent P$5^{s}$.  Let $f(X)$ be a real, bounded and measurable
  function.  Then
  \begin{align*}
    \mE\left\{f(X) \mid Y,W,Z\right\} &= \mE\left\{f(X) \mid Y,Z\right\} \,\, \text{\rm{a.s.} since \sind{X}{W}{(Y,Z)}}\\
                                      &= \mE\left\{f(X)\mid Z\right\} \,\, \text{\rm{a.s.} since
                                        \sind{X}{Y}{Z}}
  \end{align*}
  which proves P$5^{s}$.
\end{proof}

In \thmref{saxioms} we have shown that stochastic conditional
independence satisfies the axioms of a separoid.  Denoting by $V$ the
set of all random variables defined on the probability space
$(\Omega,\cA, \mP)$ and equipping $V$ with the quasiorder $\preceq$,
$(V, \preceq)$ becomes a join semilattice and the triple
$(V,\preceq,\cip)$ is then a separoid.

Using \scis\ in an axiomatic way, we can mechanically prove many
useful \cis\ results.

\begin{ex}
  Let $X,Y,Z$ be random variables on $(\Omega,\cA,\mP)$.  Then
  \sind{X}{Y}{Z} implies that \sind{(X,Z)}{Y}{Z}.
\end{ex}

\begin{proof}
  Applying P$1^{s}$ to \sind{X}{Y}{Z}, we obtain
  \begin{equation}
    \label{eq:sciup1}
    \sind{Y}{X}{Z}
  \end{equation}
  By P$2^{s}$, we obtain
  \begin{equation}
    \label{eq:sciup2}
    \sind{Y}{(X,Z)}{(X,Z)}.
  \end{equation}
  Applying P$3^{s}$ to~\eqref{sciup2}, we obtain
  \begin{equation}
    \label{eq:sciup3}
    \sind{Y}{Z}{(X,Z)}
  \end{equation}
  and applying P$5^{s}$ to~\eqref{sciup1} and~\eqref{sciup3}, we
  obtain
  \begin{equation}
    \label{eq:sciup4}
    \sind{Y}{(X,Z)}{Z}.
  \end{equation}
  The result follows by applying P$1^{s}$ to~\eqref{sciup4}.
\end{proof}

\begin{ex}[Nearest Neighbour Property of a Markov Chain]

  \noindent Let $X_1,X_2, X_3,X_4,X_5$ be random variables on
  $(\Omega,\cA,\mP)$ and suppose that
  \begin{enumerate}[label=(\roman{enumi})]
  \item\label{it:nn1} \sind{X_3}{X_1}{X_2}, \
  \item\label{it:nn2} \sind{X_4}{(X_1,X_2)}{X_3},\
  \item\label{it:nn3} \sind{X_5}{(X_1,X_2,X_3)}{X_4}.  \
  \end{enumerate}
  Then \sind{X_3}{(X_1,X_5)}{(X_2,X_4)}.
\end{ex}

\begin{proof}
  See~\citet{daw:79a}
\end{proof}

\subsection{Variation conditional independence as a separoid}

\Vci, which concerns solely non-stochastic variables, is another,
indeed much simpler, example of a separoid.

Let $\cS$ be a set with elements denoted by \eg\ $\sigma$, and let $V$
be the set of all functions with domain $\cS$ and arbitrary range
space.  The elements of $V$ will be denoted by \eg\ $X,Y,\ldots$.  We
do not require any additional properties or structure such as a
probability measure, measurability, \etc\@ We write $X \preceq Y$ to
denote that $X$ is a function of $Y$, \ie\
$Y(\sigma_1)=Y(\sigma_2) \Rightarrow X(\sigma_1)=X(\sigma_2)$.  The
equivalence classes for this quasiorder correspond to partitions of
$\cS$.  Then $(V,\preceq)$ forms a join semilattice, with join
$X \vee Y$ the function $(X,Y) \in V$.

The \emph{(unconditional) image of $Y$} is
$R(Y) := Y(\cS) = \{Y(\sigma):\sigma\in \cS\}$.  The \emph{conditional
  image of $X$, given $Y=y$} is
$R(X \mid Y=y):= \{X(\sigma): \sigma \in \cS, Y(\sigma)=y \}$.  For
simplicity of notation we will sometimes write $R(X \mid y)$ instead
of $R(X \mid Y=y)$, and $R(X \mid Y)$ for the function
$R(X \mid Y= .)$.

\begin{definition}
  We say that \textit{$X$ is variation (conditionally) independent of
    $Y$ given $Z$} (on $\Omega$) and write \vind{X}{Y}{Z} $[\cS]$ (or,
  if $\cS$ is understood, just \vind{X}{Y}{Z}) if:
$$\mbox{for any } (y,z) \in R(Y,Z),\,\, R(X \mid y,z)=R(X \mid z).$$
\end{definition}

We use the subscript $v$ under $\cip$ ($\cip_{v}$) to emphasize that
we refer to the above non-stochastic definition.  In parallel with the
stochastic case, we have equivalent forms of the above definition.

\begin{proposition}
  \label{propo:vieq}
  The following are equivalent.
  \begin{enumerate}[label=(\roman{enumi})]
  \item
    \label{it:vid1}
    \label{it:avid1}
    \vind{X}{Y}{Z}.  %
  \item\label{it:vid2}
    \label{it:avid2}
    The function $R(X \mid Y,Z)$ of $(Y,Z)$ is a function of $Z$
    alone.
  \item \label{it:vid3}
    \label{it:avid3}
    For any $z \in R(Z)$,
    $R(X,Y \mid z)=R(X \mid z)\times R(Y \mid z)$.
  \end{enumerate}
\end{proposition}

\begin{proof}
  See \appref{A}.
\end{proof}

\begin{proposition}
  \label{propo:eqlemma}
  The following are equivalent.
  \begin{enumerate}[label=(\roman{enumi})]
  \item
    \label{it:vim1}
    \label{it:avim1}
    $W \preceq Y$.
  \item
    \label{it:vim2}
    \label{it:avim2}
    there exists $f: R(Y) \rightarrow R(W)$ such that $W=f(Y)$.
  \end{enumerate}
\end{proposition}

\begin{proof}
  See \appref{A}.
\end{proof}

\begin{theorem}{(Axioms of variation independence.)}
  \label{thm:axiomsv}
  Let $X,Y,Z,W$ be functions on $\cS$.  Then the following properties
  hold.
  \begin{itemize}
  \item[P$1^{v}$.] \vind{X}{Y}{Z} $\Rightarrow$ \vind{Y}{X}{Z}.
  \item[P$2^{v}$.] \vind{X}{Y}{Y}.
  \item[P$3^{v}$.] \vind{X}{Y}{Z} and $W \preceq Y$ $\Rightarrow$
    \vind{X}{W}{Z}.
  \item[P$4^{v}$.] \vind{X}{Y}{Z} and $W \preceq Y$ $\Rightarrow$
    \vind{X}{Y}{(W,Z)}.
  \item[P$5^{v}$.] \vind{X}{Y}{Z} and \vind{X}{W}{(Y,Z)} $\Rightarrow$
    \vind{X}{Y}{(W,Z)}.
  \end{itemize}

\end{theorem}

\begin{proof}\quad\\
  P$1^{v}$.  Follows directly from \proporef{vieq}.

  \noindent P$2^{v}$.  Let $x \in R(X)$.  Then
  \begin{align*}
    R(X,Y \mid y)&:=\{(X(\sigma),Y(\sigma)): \sigma \in \cS, Y(\sigma)=y\} \\
                 &= \{(X(\sigma),y): \sigma \in \cS, Y(\sigma)=y\} \\
                 &= R(X \mid y) \times \{y\} \\
                 &= R(X \mid y) \times R(Y \mid y)
  \end{align*}
  which proves P$2^{v}$.
  
  \noindent P$3^{v}$.  Let \vind{X}{Y}{Z} and $W \preceq Y$.  Since
  $W \preceq Y$, it follows from \proporef{eqlemma} that there exists
  $f: Y(\cS) \rightarrow W(\cS)$ such that $W(\sigma)=f(Y(\sigma))$.
  For any $(w,z) \in R(W,Z)$,
  \begin{align*}
    R(X \mid w,z)&=\{X(\sigma): \sigma \in \cS, (W,Z)(\sigma)=(w,z)\} \\
                 &=\{X(\sigma): \sigma \in \cS, f(Y(\sigma))=w, Z(\sigma)=z\} \\
                 &=\{X(\sigma): \sigma \in \cS, Y(\sigma) \in f^{-1}(w), Z(\sigma)=z\} \\
                 &= \bigcup_{y \in f^{-1}(w)} \{X(\sigma): \sigma \in \cS, Y(\sigma)=y, Z(\sigma)=z\} \\
                 &= \bigcup_{y \in f^{-1}(w)} R(X \mid y,z) \\
                 &= \bigcup_{y \in f^{-1}(w)} R(X \mid z) \,\, \text{since \vind{X}{Y}{Z}} \\
                 &= R(X \mid z)
  \end{align*}
  which proves P$3^{v}$.
  
  \noindent P$4^{v}$.  Let \vind{X}{Y}{Z} and $W \preceq Y$.  Since
  $W \preceq Y$, it follows from \proporef{eqlemma} that there exists
  $f: Y(\cS) \rightarrow W(\cS)$ such that $W(\sigma)=f(Y(\sigma))$.
  Now let $(y,z,w) \in R(Y,Z,W)$.  Then $f(y)=w$ and we have that:
  \begin{align*}
    R(X \mid y,z,w)&:=\{ X(\sigma): \sigma \in \cS, (Y,Z,W)(\sigma)=(y,z,w)\} \\
                   &=\{ X(\sigma): \sigma \in \cS, Y(\sigma)=y, Z(\sigma)=z, f(Y(\sigma))=w\} \\
                   &=\{ X(\sigma): \sigma \in \cS, Y(\sigma)=y, Z(\sigma)=z \} \,\,
                     \text{since $Y(\sigma)=y \Rightarrow f(Y(\sigma))=w$}\\
                   &= R(X \mid y,z) \\
                   &= R(X \mid z) \,\, \text{since \vind{X}{Y}{Z}}.
  \end{align*}
  Now P$4^{v}$ follows from \proporef{vieq}\itref{vid2}.
  
  \noindent P$5^{v}$.  Let \vind{X}{Y}{Z} and \vind{X}{W}{(Y,Z)}.
  Also let $(y,w,z) \in R(Y,W,Z)$.  Then
  \begin{align*}
    R(X \mid y,w,z)&=R(X \mid y,z) \,\, \text{since \vind{X}{W}{(Y,Z)}} \\
                   &=R(X \mid z) \,\, \text{since \vind{X}{Y}{Z}}
  \end{align*}
  which proves P$5^{v}$.
\end{proof}

In the above theorem we have shown that variation independence
satisfies the axioms of a separoid.  Indeed---and in contrast with
stochastic conditional independence---variation independence also
satisfies the axioms of a strong separoid~\citep{daw:01b}.

\section{Extended conditional independence}
\label{sec:3}

There is a basic intuitive similarity between the notions of
stochastic conditional independence and variation independence.  A
statement like \sind{X}{Y}{Z} for stochastic variables, or
\vind{X}{Y}{Z} for non-stochastic variables, reflects our informal
understanding that, having already obtained information about $Z$,
further information about $Y$ will not affect the uncertainty
(suitably understood) about $X$.  Building on this intuitive
interpretation, one can extend \sind{X}{Y}{Z} to the case that one or
both of $Y$ and $Z$ involve non-stochastic variables, such as
parameters or regime indicators.  Such an extended version of
conditional independence would embrace the notions of ancillarity,
sufficiency, causality, \etc

The first authors to consider sufficiency in a general abstract
setting were~\citet{hal:49}.  Removing any assumption such as the
existence of a probability mass function or a density with respect to
a common measure, sufficiency is defined as follows.

\begin{definition}[Sufficiency]
  \label{def:sufficiency2}
  Consider a a random variable $X$, and a family
  $\cP = \{\mP_\theta\}$ of probability distributions for $X$, indexed
  by $\theta \in \pmb{\Theta}$.  A statistic $T=T({X})$ is
  \emph{sufficient} for $\cP$, or for $\theta$, if for any real,
  bounded and measurable function $h$, there exists a function $w(T)$
  such that, for any $\theta \in \Theta$,
  \begin{displaymath}
    \mE_{\theta}\{h(X)\mid T\}=w(T) \,\, \text{\as{\mP_{\theta}}}.
  \end{displaymath}
\end{definition}

Interpreting the definition carefully, we require that, for any real,
bounded and measurable $h(X)$, there exist a single function $w(T)$
that serves as a version of the conditional expectation
$\mE_{\theta}\{h(X) \mid T\}$ under $\mP_{\theta}$, simultaneously for
all $\theta \in \Theta$.

In the \dts\, we consider, instead of the parameter space
$\pmb{\Theta}$, a space $\cS$ of different regimes, typically
$\sigma$,\footnote{The regime indicator $\sigma$ is not to be confused
  with the $\sigma$-algebra generated by $X$, denoted by $\sigma(X)$.}
under which data can be observed.  We thus consider a family
$\cP = \{\mP_\sigma:\sigma\in\cS\}$ of probability measures over a
suitable space $(\Omega,\cA)$.  A stochastic variable, such as
$X: (\Omega,\sigma(X)) \rightarrow (\mR, \cB)$, can have different
distributions under the different regimes $\sigma \in \cS$.  We write
$\mE_\sigma(X \mid Y)$ to denote a version of the conditional
expectation $\mE(X \mid Y)$ under regime $\sigma$: this is defined
\as{\mP_\sigma}.  We also consider non-stochastic variables, functions
defined on $\cS$, which we term \emph{decision variables}.  Decision
variables give us full or partial information about which regime is
operating.  We denote by $\Sigma$ the identity function on $\cS$.

We aim to extend \defref{sufficiency2} to express a statement like
\ind{X}{(Y,\Theta)}{(Z,\Phi)}, where $X,Y,Z$ are stochastic variables
and $\Theta,\Phi$ decision variables.  In order to formalise such a
statement, we first describe what we would like a conditional
independence statement like \ind{X}{\Theta}{\Phi} to reflect
intuitively: that the distribution of $X$, given the information
carried by $(\Theta,\Phi)$ about which regime is operating, is in fact
fully determined by the value of $\Phi$ alone.

However, in order for this to make sense, we must assume that $\Phi$
and $\Theta$ together do fully determine the regime $\sigma \in S$
operating and, thus, the distribution of $X$ in this regime.
Formally, we require that the function $(\Phi,\Theta)$ defined on
$\cS$ be a surjection.  In this case we say that $\Phi$ and $\Theta$
are \emph{complementary} (on $\cS$), or that \emph{$\Theta$ is
  complementary to $\Phi$} (on $\cS$).  This leads to the following
definition.

\begin{definition}
  \label{def:nstocdef1}
  Let $X$, $Y$ and $Z$ be stochastic variables, and let $\Phi$ and
  $\Theta$ be complementary decision variables.  We say that \emph{$X$
    is (conditionally) independent of $(Y,\Theta)$ given $(Z,\Phi)$}
  and write \ind{X}{(Y,\Theta)}{(Z,\Phi)} if, for all
  $\phi \in \Phi(\cS)$ and all real, bounded and measurable functions
  $h$, there exists a function $w_{\phi}(Z)$ such that, for all
  $\sigma \in \Phi^{-1}(\phi)$,
  \begin{equation}
    \label{eq:nstocdefeq}
    \mE_\sigma\{h(X)\mid Y,Z\}=w_{\phi}(Z) \,\, \text{\as{\mP_\sigma}}.
  \end{equation}
\end{definition}

We will refer to this definition of conditional independence as
\emph{extended conditional independence}.  Note that the only
important property of $\Theta$ in the above definition is that it is
complementary to $\Phi$; beyond this, the actual form of $\Theta$
becomes irrelevant.  Henceforth, we will write down a conditional
independence statement involving two decision variables only when the
two variables are complementary.

\begin{remark}
  \label{rem:nstoc1}
  Assume that \ind{X}{(Y,\Theta)}{(Z,\Phi)} and consider $w_{\phi}(Z)$
  as in \defref{nstocdef1}.  Then
  \begin{align*}
    \mE_\sigma\left\{ h(X) \mid Z \right \} &= \mE_\sigma\left[
                                              \mE_\sigma \left\{ h(X) \mid Y,Z \right\} \mid Z \right]
                                              \,\, \text{a.s.\ $\mathbb{P}_\sigma$} \\
                                            &=\mE_\sigma \left\{ w_{\phi}(Z) \mid Z \right \} \,\, \text{a.s.\ $\mathbb{P}_\sigma$} \\
                                            &=w_{\phi}(Z) \,\, \text{a.s.\ $\mathbb{P}_\sigma$.}
  \end{align*}
  Thus $w_{\phi}(Z)$ also serves as a version of
  $\mE_\sigma\{h(X) \mid Z\}$ for all $\sigma \in \Phi^{-1}(\phi)$.
\end{remark}

The following example shows that, even when~\eqref{nstocdefeq} holds,
we can not use just any version of the conditional expectation in one
regime to serve as a version of the conditional expectation in another
regime.  This is because two versions of the conditional expectation
can differ on a set of probability zero, but a set of probability zero
in one regime could have positive probability in another.
\begin{ex}
  \label{ex:nstocdef}
  Let the regime space be $S=\{\sigma_{0},\sigma_{1}\}$, let binary
  variable $T$ represent the treatment taken (where $T=0$ denotes
  placebo and $T=1$ denotes active treatment), and let $X$ be an
  outcome of interest.  Regime $\sigma_t$ ($t=0,1$) represents the
  interventional regime under treatment $t$: in particular,
  $\mP_{\sigma_{j}}(T=j)=1$.

  We consider the situation where the treatment is ineffective, so
  that $X$ has the same distribution in both regimes.  We then have
  $\ind{X}{\Sigma}{T}$ --- since, for any function $h(X)$, we can take
  as $\mE\{h(X) \mid \sigma, T\}$, for both $\sigma =0$ and
  $\sigma=1$, the (constant) common expectation of $h(X)$ in both
  regimes.\footnote{Indeed, this encapsulates the still stronger
    property $\indo{X}{(\Sigma, T)}$.}

  In particular, suppose $X$ has expectation 1 in both regimes.  Then
  the function $w(T) \equiv 1$ is a version both of
  $\mE\{h(X) \mid \sigma_0, T\}$ and of
  $\mE\{h(X) \mid \sigma_1, T\}$.  That is,
  $\mE_{\sigma_{0}}(X \mid T)=1$ \as{\mP_{\sigma_{0}}}, and
  $\mE_{\sigma_{1}}(X \mid T)=1$ \as{\mP_{\sigma_{1}}}.

  Now consider the functions
  \begin{equation*}
    k_0(t) = 1-t \qquad \text{and} \qquad  k_1(t) = t.
  \end{equation*}
  
  We can see that $k_0(T)=w(T)$ \as{\mP_{\sigma_{0}}}, so that
  $k_0(T)$ is a version of $\mE\{h(X) \mid \sigma_0, T\}$; similarly,
  $k_1(T)$ is a version of $\mE\{h(X) \mid \sigma_1, T\}$.  However,
  almost surely, under both $\mP_{\sigma_{0}}$ and $\mP_{\sigma_{1}}$,
  $k_{0}(T) \neq k_{1}(T)$.  Hence neither of these variables can
  replace $w(T)$ in supplying a version of $\mE_\sigma(X \mid T)$
  simultaneously in both regimes.
\end{ex}
 
We now introduce some equivalent versions of \defref{nstocdef1}.

\begin{proposition}
  \label{propo:nsprop1}
  Let $X,Y,Z$ be stochastic variables and let $\Phi,\Theta$ be
  decision variables.  Then the following are equivalent.
  \begin{enumerate}[label=(\roman{enumi})]
  \item
    \label{it:nspr1}
    \ind{X}{(Y,\Theta)}{(Z,\Phi)}.
  \item
    \label{it:nspr2}
    For all $\phi \in \Phi(\cS)$ and all real, bounded and measurable
    function $h_{1}$, there exists a function $w_{\phi}(Z)$ such that,
    for all $\sigma \in \Phi^{-1}(\phi)$ and all real, bounded and
    measurable functions $h_{2}$,
    \begin{equation*}
      \mE_\sigma\left\{h_{1}(X)h_{2}(Y)\mid Z\right\}=w_{\phi}(Z)\,
      \mE_\sigma\left\{h_{2}(Y)\mid Z\right\} \,\, \text{\as{\mP_\sigma}}.  
    \end{equation*}
  \item
    \label{it:nspr3}
    For all $\phi \in \Phi(\cS)$ and all $A_{X} \in \sigma(X)$, there
    exists a function $w_{\phi}(Z)$ such that, for all
    $\sigma \in \Phi^{-1}(\phi)$ and all $A_{Y} \in \sigma(Y)$,
    \begin{equation*}
      \mE_\sigma(\mathbbm{1}_{A_{X}\cap A_{Y}}\mid Z)=w_{\phi}(Z) \,
      \mE_\sigma(\mathbbm{1}_{A_{Y}}\mid Z) \,\, \text{\as{\mP_\sigma}}.
    \end{equation*} 
  \item
    \label{it:nspr4}
    For all $\phi \in \Phi(\cS)$ and all $A_{X} \in \sigma(X)$, there
    exists a function $w_{\phi}(Z)$ such that, for all
    $\sigma \in \Phi^{-1}(\phi)$,
    \begin{equation}
      \label{eq:nstoceq}
      \mE_\sigma(\mathbbm{1}_{A_{X}}\mid Y,Z)=w_{\phi}(Z) \,\, \text{\as{\mP_\sigma}}.
    \end{equation}
  \end{enumerate}
\end{proposition}

\begin{proof}
  See \appref{A}.
\end{proof}

Using \proporef{nsprop1} we can obtain further properties of \eci.
For example, we can show that \defref{nstocdef1} can be equivalently
expressed in two simpler statements of \eci, or that when all the
decision variables are confined to the right-most term symmetry does
follow.  In \secref{asa} we will show still more properties.

\begin{proposition}
  \label{propo:defbreak}
  Let $X$, $Y$, $Z$ be stochastic variables and $\Phi$, $\Theta$
  decision variables.  Then the following are equivalent:
  \begin{enumerate}[label=(\roman{enumi})]
  \item
    \label{it:bb1}
    \ind{X}{(Y,\Theta)}{(Z,\Phi)}
  \item
    \label{it:bb2}
    \ind{X}{Y}{(Z,\Phi,\Theta)} and \ind{X}{\Theta}{(Z,\Phi)}.
  \end{enumerate}
\end{proposition}
 
\begin{proof}\quad\\
  \itref{bb1} $\Rightarrow$\itref{bb2}.  Since
  \ind{X}{(Y,\Theta)}{(Z,\Phi)}, for all $\phi \in \Phi(\cS)$ and
  $A_{X} \in \sigma(X)$, there exists $w_{\phi}(Z)$ such that for all
  $\sigma \in \Phi^{-1}(\phi)$,
  \begin{equation*}
    \mE_\sigma\left(1_{A_{X}} \mid Y,Z\right)=w_{\phi}(Z) \,\, \text{\as{\mP_\sigma}}
  \end{equation*}
  which proves that \ind{X}{Y}{(Z,\Phi,\Theta)}.  Also, by
  \remref{nstoc1},
  \begin{equation*}
    \mE_\sigma\left(1_{A_{X}} \mid Z\right)=w_{\phi}(Z) \,\, \text{\as{\mP_\sigma}}
  \end{equation*}
  which proves that \ind{X}{\Theta}{(Z,\Phi)}.

  \noindent \itref{bb2} $\Rightarrow$ \itref{bb1}.  Since
  \ind{X}{Y}{(Z,\Phi,\Theta)}, for all $\sigma \in \cS$ and
  $A_{X} \in \sigma(X)$, there exists $w_\sigma(Z)$ such that
  \begin{equation}
    \label{eq:defbreak1}
    \mE_\sigma[1_{A_{X}} \mid Y,Z]= w_\sigma(Z) \,\, \text{\as{\mP_\sigma}}.  
  \end{equation}
  By \remref{nstoc1},
  \begin{equation}
    \label{eq:defbreak2}
    \mE_\sigma\left(1_{A_{X}} \mid Z\right)= w_\sigma(Z) \,\, \text{\as{\mP_\sigma}}.  
  \end{equation}
  Since \ind{X}{\Theta}{(Z,\Phi)}, for all $\phi \in \Phi(\cS)$ and
  $A_{X} \in \sigma(X)$ there exists $w_{\phi}(Z)$ such that, for all
  $\sigma \in \Phi^{-1}(\phi)$,
  \begin{equation}
    \label{eq:defbreak3}
    \mE_\sigma\left(1_{A_{X}} \mid Z\right)=w_{\phi}(Z) \,\, \text{\as{\mP_\sigma}}.
  \end{equation}
  By~\eqref{defbreak2} and~\eqref{defbreak3},
  \begin{equation*}
    w_\sigma(Z)=w_{\phi}(Z) \,\, \text{\as{\mP_\sigma}}.
  \end{equation*}
  Thus, by~\eqref{defbreak1},
  \begin{equation*}
    \mE_\sigma\left(1_{A_{X}} \mid Y,Z\right)= w_{\phi}(Z) \,\, \text{\as{\mP_\sigma}},
  \end{equation*}
  which proves that \ind{X}{(Y,\Theta)}{(Z,\Phi)}.
\end{proof}

\begin{proposition}
  \label{propo:nstocsym}
  Let $X$, $Y$, $Z$ be stochastic variables, and $\Sigma$ a decision
  variable.  Then \ind{X}{Y}{(Z,\Sigma)} if and only if \sind{X}{Y}{Z}
  under $\mP_\sigma$ for all $\sigma \in\cS$.
\end{proposition}

\begin{proof}\quad\\
  \ind{X}{Y}{(Z,\Sigma)} is equivalent to: for all $\sigma \in \cS$
  and all $A_{X} \in \sigma(X)$ there exists a function $w_\sigma(Z)$
  such that, for all $A_{Y} \in \sigma(Y)$,
  \begin{equation*}
    \mE_\sigma\left(\mathbbm{1}_{A_{X}\cap A_{Y}}\mid Z\right)
    =w_\sigma(Z) \mE_\sigma(\mathbbm{1}_{A_{Y}}\mid Z) \,\, \text{\as{\mP_\sigma}}.
  \end{equation*} 
  In particular, for $A_Y=\Omega$ we have
  \begin{equation*}
    \mE_\sigma\left(\mathbbm{1}_{A_{X}}\mid Z\right)=w_\sigma(Z) \,\, \text{\as{\mP_\sigma}}.
  \end{equation*}
  The property \ind{X}{Y}{(Z,\Sigma)} is thus equivalent to: for all
  $\sigma \in\cS$, $A_{X} \in \sigma(X)$, $A_{Y} \in \sigma(Y)$,
  \begin{equation*}
    \mE_\sigma\left(\mathbbm{1}_{A_{Y}\cap A_{X}}\mid Z\right)
    =\mE_\sigma\left(\mathbbm{1}_{A_{Y}}\mid Z\right) \mE_\sigma\left(\mathbbm{1}_{A_{X}}\mid Z\right)
    \,\, \text{\as{\mP_\sigma}},
  \end{equation*} 
  which concludes the proof.
\end{proof}
\begin{corollary}
  \label{cor:nstocsym}
  \ind{X}{Y}{(Z,\Sigma)} $\Rightarrow \ind{Y}{X}{(Z,\Sigma)}$.
\end{corollary}

\subsection{Some separoid properties}
\label{sec:asa}

Comparing \defref{nstocdef1} for extended conditional independence
with \defref{scidef} for stochastic conditional independence, we
observe a close technical, as well as intuitive, similarity.  This
suggests that these two concepts should have similar properties, and
motivates the conjecture that the separoid axioms of conditional
independence will continue to hold for the extended concept.  In this
Section we show that this is indeed so, in complete generality, for a
subset of the axioms.  However, in order to extend this to other
axioms we need to impose additional conditions---this we shall develop
in later Sections.

One important difference between \eci\ and \sci\ concerns the symmetry
axiom P1.  Whereas symmetry holds universally for \sci, its
application to \eci\ is constrained by the fact that, for
\defref{nstocdef1} even to make sense, the first term $x$ in an \eci\
relation of the form \ind{x}{y}{z} must be fully stochastic, whereas
the second term $y$ can contain a mixture of stochastic and
non-stochastic variables---in which case it would make no sense to
interchange $x$ and $y$.  This restricted symmetry also means that
each of the separoid axioms P2--P5 has a possibly non-equivalent
``mirror image'' version, obtained by interchanging the first and
second terms in each relation.

The following theorem demonstrates certain specific versions of the
separoid axioms.

\begin{theorem}
  \label{thm:nstocaxioms}
  Let $X$, $Y$, $Z$, $W$ be stochastic variables and $\Phi$, $\Theta$,
  $\Sigma$ be decision variables.  Then the following properties hold.
  \begin{itemize}
  \item[P$1'$.] \ind{X}{Y}{(Z,\Sigma)} $\Rightarrow$
    \ind{Y}{X}{(Z,\Sigma)}.
  \item[P$2'$.] \ind{X}{(Y,\Sigma)}{(Y,\Sigma)}.
  \item[P$3'$.] \ind{X}{(Y,\Theta)}{(Z,\Phi)} and $W \preceq Y$
    $\Rightarrow$ \ind{X}{(W,\Theta)}{(Z,\Phi)}.
  \item[P$4'$.] \ind{X}{(Y,\Theta)}{(Z,\Phi)} and $W \preceq Y$
    $\Rightarrow$ \ind{X}{(Y,\Theta)}{(Z,W,\Phi)}.
  \item[P$5'$.] \ind{X}{(Y,\Theta)}{(Z,\Phi)} and
    \ind{X}{W}{(Y,Z,\Theta,\Phi)} $\Rightarrow$
    \ind{X}{(Y,W,\Theta)}{(Z,\Phi)}.
  \end{itemize}
\end{theorem}

\begin{proof} \quad\\
  P$1'$.  Proved in \proporef{nstocsym}.

  \noindent P$2'$.  Let $\sigma \in \cS$ and $A_{X} \in \sigma(X)$.
  Then for all $A_{Y} \in \sigma(Y)$,
  \begin{equation*}
    \mE_\sigma(\mathbbm{1}_{A_{X} \cap A_{Y}} \mid Y)= \mathbbm{1}_{A_{Y}}\mE_\sigma(\mathbbm{1}_{A_{X}} \mid Y) \,\, \text{\as{\mP_\sigma}}
  \end{equation*}
  which concludes the proof.
  
  \noindent P$3'$.  Let $\phi \in \Phi(\cS)$ and
  $A_{X} \in \sigma(X)$.  Since \ind{X}{(Y,\Theta)}{(Z,\Phi)}, there
  exists $w_{\phi}(Z)$ such that, for all
  $\sigma \in \Phi^{-1}(\phi)$,
  \begin{equation*}
    \mE_\sigma\left(\mathbbm{1}_{A_{X}} \mid Y,Z\right)=w_{\phi}(Z) \,\, \text{\as{\mP_\sigma}}.
  \end{equation*}
  Since $W \preceq Y$, it follows from \lemref{measurability} that
  $\sigma(W) \subseteq \sigma(Y)$ and thus
  $\sigma(W,Z) \subseteq \sigma(Y,Z)$.  Then
  \begin{align*}
    \mE_\sigma\left(\mathbbm{1}_{A_{X}} \mid W,Z\right)
    &=\mE_\sigma\left\{\mE_\sigma(\mathbbm{1}_{A_{X}} \mid Y,Z) \mid
      W,Z\right\}
      \,\, \text{\as{\mP_\sigma}} \\
    &=\mE_\sigma\left\{w_{\phi}(Z)\mid W,Z\right\} \,\, \text{\as{\mP_\sigma}} \\
    &=w_{\phi}(Z)\,\, \text{\as{\mP_\sigma}}
  \end{align*}
  which concludes the proof.
  
  \noindent P$4'$.  Let $\phi \in \Phi(\cS)$ and
  $A_{X} \in \sigma(X)$.  Since \ind{X}{(Y,\Theta)}{(Z,\Phi)}, there
  exists $w_{\phi}(Z)$ such that, for all
  $\sigma \in \Phi^{-1}(\phi)$,
  \begin{equation*}
    \mE_\sigma\left(\mathbbm{1}_{A_{X}} \mid Y,Z\right) 
    =w_{\phi}(Z) \,\, \text{\as{\mP_\sigma}}.
  \end{equation*}
  Since $W \preceq Y$, it follows from \lemref{measurability} that
  $\sigma(W) \subseteq \sigma(Y)$ and thus
  $\sigma(Y,Z,W)=\sigma(Y,Z)$.  Then
  \begin{align*}
    \mE_\sigma\left(\mathbbm{1}_{A_{X}} \mid Y,Z,W\right)
    &=\mE_\sigma\left(\mathbbm{1}_{A_{X}} \mid Y,Z\right)
      \,\, \text{\as{\mP_\sigma}} \\
    &=w_{\phi}(Z)\,\, \text{\as{\mP_\sigma}}
  \end{align*}
  which concludes the proof.

  \noindent P$5'$.  Let $\phi \in \Phi(\cS)$ and
  $A_{X} \in \sigma(X)$.  Since \ind{X}{(Y,\Theta)}{(Z,\Phi)}, there
  exists $w_{\phi}(Z)$ such that, for all
  $\sigma \in \Phi^{-1}(\phi)$,
  \begin{equation*}
    \mE_\sigma\left(\mathbbm{1}_{A_{X}} \mid Y,Z\right) 
    =w_{\phi}(Z) \,\, \text{\as{\mP_\sigma}}.
  \end{equation*}
  Since $W \preceq Y$, it follows from \lemref{measurability} that
  $\sigma(W) \subseteq \sigma(Y)$ and thus
  $\sigma(Y,W,Z) = \sigma(Y,Z)$.  Then
  \begin{align*}
    \mE_\sigma\left(\mathbbm{1}_{A_{X}} \mid Y,W,Z\right) &=
                                                            \mE_\sigma\left(\mathbbm{1}_{A_{X}} \mid Y,Z\right)
                                                            \,\, \text{\as{\mP_\sigma}} \\
                                                          &= w_{\phi}(Z) \,\, \text{\as{\mP_\sigma}}
  \end{align*}
  which concludes the proof.
\end{proof}

Lack of symmetry however, introduces some complications as the
symmetric equivalents of axioms P$3'$, P$4'$ and P$5'$ do not
automatically follow.

Consider the following statements, which mirror P$3'$--P$5'$:
\begin{enumerate}[label=(\roman{enumi})]
\item[P$3''$.]\ind{X}{(Y,\Theta)}{(Z,\Phi)} and $W \preceq X$
  $\Rightarrow$ \ind{W}{(Y,\Theta)}{(Z,\Phi)}.
\item[P$4''$.]\ind{X}{(Y,\Theta)}{(Z,\Phi)} and $W \preceq X$
  $\Rightarrow$ \ind{X}{(Y,\Theta)}{(Z,W,\Phi)}.
\item[P$5''$.]\ind{X}{(Y,\Theta)}{(Z,\Phi)} and
  \ind{W}{(Y,\Theta)}{(X,Z,\Phi)} $\Rightarrow$
  \ind{(X,W)}{(Y,\Theta)}{(Z,\Phi)}.
\end{enumerate}
P$3''$ follows straightforwardly, and P$5''$ will be proved to hold in
\proporef{P5''} below.  However, P$4''$ presents some difficulty.  We
will see below (\corref{nstoc2}, \proporef{nstocaxiom4disc},
\proporef{nstocaxiom4}) that we can obtain P$4''$ under certain
additional conditions, but validity in full generality remains an open
problem.

\begin{lemma}
  \label{lem:nstocaxiom3i}
  Let $X$, $Y$, $Z$, $W$ be stochastic variables and $\Phi$, $\Theta$
  be decision variables.  Then
  \begin{equation*}
    \ind{X}{(Y,\Theta)}{(Z,\Phi)} \text{ and } W \preceq X 
    \Rightarrow \ind{(W,Z)}{(Y,\Theta)}{(Z,\Phi)}.
  \end{equation*}
\end{lemma}

\begin{proof}
  Since \ind{X}{(Y,\Theta)}{(Z,\Phi)}, for all $\phi \in \Phi(\cS)$
  and all $A_{X} \in \sigma(X)$ there exists $w_{\phi}(Z)$ such that,
  for all $\sigma \in \Phi^{-1}(\phi)$,
  \begin{equation}
    \label{eq:nstocaxiom3ia}
    \mE_\sigma\left(\mathbbm{1}_{A_{X}} \mid Y,Z\right)=w_{\phi}(Z) \,\, \text{\as{\mP_\sigma}}.
  \end{equation}
  To prove that \ind{(W,Z)}{(Y,\Theta)}{(Z,\Phi)}, let
  $\phi \in \Phi(\cS)$ and $A_{W,Z} \in \sigma(W,Z)$.  We will show
  that there exists $a_{\phi}(Z)$ such that, for all
  $\sigma \in \Phi^{-1}(\phi)$,
  \begin{equation}
    \label{eq:nstocaxiom3ib}
    \mE_\sigma\left(\mathbbm{1}_{A_{W,Z}} \mid Y,Z\right)=a_{\phi}(Z) \,\, \text{\as{\mP_\sigma}}.
  \end{equation}
  Consider
  \begin{equation*}
    \cD= \{A_{W,Z} \in \sigma(W,Z): \text{there exists $a_{\phi}(Z)$ 
      such that~\eqref{nstocaxiom3ib} holds}\}
  \end{equation*}
  and
  \begin{equation*}
    \Pi= \{A_{W,Z} \in \sigma(W,Z): A_{W,Z}=A_{W} \cap A_{Z} \text{ for } A_{W} \in \sigma(W) \text{ and } A_{Z} \in \sigma(Z)\}.
  \end{equation*}
  
  Then $\sigma(\Pi)=\sigma(W,Z)$~\cite[p.~73]{res:14}.  We will show
  that $\cD$ is a $d$-system that contains $\Pi$.  We can then apply
  Dynkin's lemma~\cite[p.~42]{bil:95} to conclude that $\cD$ contains $\sigma(\Pi)=\sigma(W,Z)$.\\ \\
  To show that $\cD$ contains $\Pi$, let $A_{W,Z}=A_{W}\cap A_{Z}$
  with $A_{W} \in \sigma(W)$ and $A_{Z} \in \sigma(Z)$.  Then
  \begin{align*}
    \mE_\sigma\left(\mathbbm{1}_{A_{W}}\mathbbm{1}_{A_{Z}} \mid
    Y,Z\right) &=\mathbbm{1}_{A_{Z}}
                 \mE_\sigma\left(\mathbbm{1}_{A_{W}} \mid Y,Z\right)
                 \,\, \text{\as{\mP_\sigma}} \\
               &=\mathbbm{1}_{A_{Z}} w_{\phi}(Z) \,\, \text{\as{\mP_\sigma} 
                 by~\eqref{nstocaxiom3ia}}.
  \end{align*}
  Now define $a_{\phi}(Z):=\mathbbm{1}_{A_{Z}} w_{\phi}(Z)$ and we are done.  \\ \\
  To show that $\cD$ is a $d$-system, first note that
  $\Omega \in \cD$.  Also, for $A_1,A_2 \in \cD$ such that
  $A_1 \subseteq A_2$, we can readily see that
  $A_2 \setminus A_1 \in \cD$.  Now consider an increasing sequence
  $(A_{n}: n \in \mN)$ in $\cD$ and denote by $a_{\phi}^{A_n}(Z)$ the
  corresponding function such that~\eqref{nstocaxiom3ib} holds.  Then
  $A_{n} \uparrow \displaystyle\cup_{n}A_{n}$ and
  $\mathbbm{1}_{A_{n}} \uparrow
  \mathbbm{1}_{\displaystyle\cup_{n}A_{n}}$ pointwise.  Thus, by
  conditional monotone convergence~\cite[p.~193]{dur:13},
  \begin{align*}
    \mE_\sigma\left(\mathbbm{1}_{\displaystyle\cup_{n}A_{n}} \mid
    Y,Z\right) &= \lim_{n \rightarrow
                 \infty}\mE_\sigma\left(\mathbbm{1}_{A_{n}} \mid Y,Z\right)
                 \,\, \text{\as{\mP_\sigma}}  \\
               &= \lim_{n \rightarrow \infty} a_{\phi}^{A_n}(Z) \,\,
                 \text{\as{\mP_\sigma}}.
  \end{align*}
  Now define
  $a_{\phi}(Z):=\displaystyle\lim_{n \rightarrow \infty}
  a_{\phi}^{A_n}(Z)$ and we are done.
\end{proof}

\begin{proposition}
  \label{propo:P5''}
  Let $X$, $Y$, $Z$, $W$ be stochastic variables and $\Phi$, $\Theta$
  decision variables.  Then\\[1ex]
  \noindent P$5''$: \begin{math} \ind{X}{(Y,\Theta)}{(Z,\Phi)} \text{
      and } \ind{W}{(Y,\Theta)}{(X,Z,\Phi)} \Rightarrow
    \ind{(X,W)}{(Y,\Theta)}{(Z,\Phi)}.
  \end{math}
\end{proposition}

\begin{proof}
  Following the same approach as in the proof of
  \lemref{nstocaxiom3i}, to prove that
  \ind{(X,W)}{(Y,\Theta)}{(Z,\Phi)} it is enough to show that, for all
  $\phi \in \Phi(\cS)$ and all $A_{X,W}=A_{X} \cap A_{W}$ where
  $A_{X} \in \sigma(X)$ and $A_{W} \in \sigma(W)$, there exists
  $w_{\phi}(Z)$ such that, for all $\sigma \in \Phi^{-1}(\phi)$,
  \begin{equation}
    \label{eq:P5''i}
    \mE_\sigma\left(\mathbbm{1}_{A_{X,W}} \mid Y,Z\right)=w_{\phi}(Z) \,\, \text{\as{\mP_\sigma}}.
  \end{equation}
  Since \ind{W}{(Y,\Theta)}{(X,Z,\Phi)}, for all $\phi \in \Phi(\cS)$
  and all $A_{W} \in \sigma(W)$ there exists $w^{1}_{\phi}(X,Z)$ such
  that, for all $\sigma \in \Phi^{-1}(\phi)$,
  \begin{equation}
    \label{eq:P5''ii}
    \mE_\sigma\left(\mathbbm{1}_{A_{W}} \mid X,Y,Z\right)=w^{1}_{\phi}(X,Z) \,\, \text{\as{\mP_\sigma}}.
  \end{equation}
  Also by \lemref{nstocaxiom3i},
  \begin{equation*}
    \ind{X}{(Y,\Theta)}{(Z,\Phi)} \Rightarrow \ind{(X,Z)}{(Y,\Theta)}{(Z,\Phi)}.
  \end{equation*}
  Thus, for all $\phi \in \Phi(\cS)$ and all $h(X,Z)$, there exists
  $w^{2}_{\phi}(Z)$ such that, for all $\sigma \in \Phi^{-1}(\phi)$,
  \begin{equation}
    \label{eq:P5''iii}
    \mE_\sigma\left\{h(X,Z) \mid Y,Z\right\}=w^{2}_{\phi}(Z) 
    \,\, \text{\as{\mP_\sigma}}.
  \end{equation}
  Let $\phi \in \Phi(\cS)$ and $A_{X,W}=A_{X} \cap A_{W}$, where
  $A_{X} \in \sigma(X)$ and $A_{W} \in \sigma(W)$.  Then
  \begin{align*}
    \mE_\sigma\left(\mathbbm{1}_{A_{X} \cap A_{W}} \mid Y,Z\right) &=
                                                                     \mE_\sigma\left\{\mE_\sigma\left(\mathbbm{1}_{A_{X} \cap A_{W}}
                                                                     \mid X,Y,Z\right) \mid Y,Z\right\}
                                                                     \,\, \text{\as{\mP_\sigma}}\\
                                                                   &= \mE_\sigma\left\{\mathbbm{1}_{A_{X}}
                                                                     \mE_\sigma\left(\mathbbm{1}_{A_{W} }\mid X,Y,Z\right) \mid
                                                                     Y,Z\right\}
                                                                     \,\, \text{\as{\mP_\sigma}}\\
                                                                   &= \mE_\sigma\left\{\mathbbm{1}_{A_{X}} w^{1}_{\phi}(X,Z) \mid
                                                                     Y,Z\right\}
                                                                     \,\, \text{\as{\mP_\sigma} by~\eqref{P5''ii}} \\
                                                                   &= w^{2}_{\phi}(Z) \,\, \text{\as{\mP_\sigma} by~\eqref{P5''iii}},
  \end{align*}
  which proves~\eqref{P5''i}.
\end{proof}

\section{A Bayesian approach}
\label{sec:bayes}
In the present Section, we introduce a Bayesian construction in an
attempt to justify the remaining separoid axioms.  We extend the
original space in order to construe both stochastic and non-stochastic
variables as measurable functions on the new space and create an
analogy between \ecis\ and \sci.  Similar ideas can be found in a
variety of contexts in probability theory and statistics.  Examples
include Poisson random processes~\citep[pp.~82--84]{kin:93}, or
Bayesian approaches to statistics~\citep{kol:42}.  We will see that,
under the assumption of a discrete regime space, \ecis\ and \scis\ are
equivalent.  Thus we can continue to apply all the properties
P$1^{s}$--P$5^{s}$ of \thmref{saxioms}.

Consider a measurable space $(\Omega,\cA)$ and a regime space $\cS$
and let $\cF$ be a $\sigma$-algebra of subsets of $\cS$.  We can
expand the original space $(\Omega,\cA)$ and consider the product
space $\Omega \times \cS$ with its corresponding $\sigma$-algebra
$\cA \otimes \cF$, where
$\cA \otimes \cF := \sigma(\cA \times \cF):=\sigma(\{A \times B: A \in
\cA, B \in \cF\})$.  Thus, we can regard all stochastic variables
$X, Y, Z, \ldots$ defined on $(\Omega,\cA)$ also as defined on
$(\Omega \times \cS,\cA \otimes \cF)$ and all $\cF$-measurable
decision variables $\Theta, \Phi, \ldots$ defined on $\cS$ also as
defined on $(\Omega \times \cS,\cA \otimes \cF)$.  To see this,
consider any stochastic variable
$X:(\Omega, \cA) \rightarrow (\mR, \cB_{X})$.  For any such $X$ we can
define
$X^*:(\Omega \times \cS,\cA \otimes \cF) \rightarrow (\mR, \cB_{X})$
by $X^{*}(\omega,\sigma)=X(\omega)$.  It is readily seen that $X^{*}$
is $\cA \otimes \cF$-measurable.  Similar justification applies for
decision variables.  Thus, in the initial space $(\Omega,\cA)$ we can
talk about \ecis\ and in the product space
$(\Omega \times \cS,\cA \otimes \cF)$, after we equip it with a
probability measure, we can talk about \sci.  To rigorously justify
the equivalence of \ecis\ and \sci, we will need the following
results.
 
\begin{lemma}
  \label{lem:product1}
  Let $f: \Omega \times \cS \rightarrow \mR$ be
  $\cA \otimes \cF$-measurable.  Define for all $\sigma \in \cS$,
  $f_\sigma: \Omega \rightarrow \mR$ by
  $f_\sigma(\omega):=f(\omega,\sigma)$.  Then $f_\sigma$ is
  $\cA$-measurable.  If further $f$ is bounded, define for all
  $\sigma \in \cS$, $\mE_\sigma(f_\sigma): \cS \rightarrow \mR$ by
  $\mE_\sigma(f_\sigma):= \int_{\Omega} f_\sigma(\omega)
  \mP_\sigma(d\omega)$.  Then the function
  $\sigma \mapsto \mE_\sigma(f_\sigma)$ is bounded and
  $\cF$-measurable.
\end{lemma}

\begin{proof}
  See~\citet[p.~231,~Theorem~18.1, and p.~234,~Theorem~18.3]{bil:95}.
\end{proof}

Now let $\pi$ be a probability measure on $(F,\cF)$.  For
$A^{*} \in \cA \otimes \cF$, define
\begin{equation}
  \mP^{*}(A^{*}):= \int_{\cS} \int_{\Omega} \mathbbm{1}_{A^{*}}(\omega,\sigma) \mP_\sigma(d\omega) \pi(d\sigma).
\end{equation}

\begin{theorem}
  \label{thm:thp}
  $\mP^{*}$ is the unique probability measure on $\cA \otimes \cF$
  such that
  \begin{equation}
    \label{eq:thp1}
    \mP^{*}(A \times B)=  \int_{B} \mP_\sigma(A) \pi(d\sigma)
  \end{equation}
  for all $A \in \cA$ and $B \in \cF$.
\end{theorem}

\begin{proof}
  By \lemref{product1}, $\mP^{*}: \cA \otimes \cF \rightarrow [0,1]$
  is a well defined function.  To prove that $\mP^{*}$ is a measure,
  we need to prove countable additivity.  So let $(A_{n}: n \in \mN)$
  be a sequence of disjoint sets in $\cA \otimes \cF$ and define
  $B_{n}:= \cup_{k=1}^{n}A_{k}$ an increasing sequence of sets.  By
  monotone convergence theorem, for each $\sigma \in \cS$, as
  $n \rightarrow \infty$,
  \begin{equation*}
    \int_{\Omega} \mathbbm{1}_{B_{n}}(\omega, \sigma) \mP_\sigma(d\omega) \uparrow 
    \int_{\Omega} \mathbbm{1}_{\cup_{k}B_{k}}(\omega, \sigma) \mP_\sigma(d\omega),
  \end{equation*}
  and hence
  \begin{equation}
    \label{eq:thp2}
    \int_{\cS}\int_{\Omega} \mathbbm{1}_{B_{n}}(\omega, \sigma) \mP_\sigma(d\omega)\pi(d\sigma) \uparrow 
    \int_{\cS}\int_{\Omega} \mathbbm{1}_{\cup_{k}B_{k}}(\omega, \sigma) \mP_\sigma(d\omega)\pi(d\sigma).
  \end{equation}
  Thus
  \begin{align*}
    \mP^{*}\left(\bigcup_{n}A_{n}\right)&= \mP^{*}\left(\bigcup_{n}B_{n}\right)  \\
                                        &= \lim_{n \rightarrow \infty} \mP^{*}(B_{n}) \,\, \text{by~\eqref{thp2}} \\
                                        &= \lim_{n \rightarrow \infty} \int_{\cS}\int_{\Omega}
                                          \sum_{k=1}^{n}{\mathbbm{1}_{A_{n}}} (\omega, \sigma)
                                          \mP_\sigma(d\omega)\pi(d\sigma) \,\, \text{since $A_{n}$ disjoint}\\
                                        &= \sum_{n}{\mP^{*}(A_{n})}.
  \end{align*}
  We can readily see that $\mP^{*}$ is a probability measure and,
  since $\mathbbm{1}_{A \times B}= \mathbbm{1}_{A}\mathbbm{1}_{B}$,
  property~\eqref{thp1} holds for all $A \in \cA$ and $B \in \cF$.
  Since $\cA \times \cF := \{A \times B: A \in \cA, B \in \cF\}$ is a
  $\pi$-system generating $\cA \otimes \cF$ and
  $\mP^{*}(\Omega \times \cS)=1 < \infty$, $\mP^{*}$ is uniquely
  determined by its values on $\cA \times \cF$, by the uniqueness of
  extension theorem~\citep[p.~42]{bil:95}.
\end{proof}

\begin{theorem}
  Let $f: \Omega \times \cS \rightarrow \mR$ be an
  $\cA \otimes \cF$-measurable integrable function.  Then
  \begin{equation}
    \label{eq:thpf}
    \mE^{*}(f)= \int_{\cS} \int_{\Omega} f(\omega,\sigma) \mP_\sigma(d\omega) \pi(d\sigma).
  \end{equation}
\end{theorem}

\begin{proof}
  Since $f$ is integrable, $\mE^{*}(f)=\mE^{*}(f^{+})-\mE^{*}(f^{-})$.
  Thus, it is enough to show~\eqref{thpf} for non-negative
  $\cA \otimes \cF$-measurable functions.  By definition of $\mP^{*}$
  in \thmref{thp}, \eqref{thpf} holds for all $f=\mathbbm{1}_{A}$,
  where $A \in \cA \otimes \cF$.  By linearity of the integrals, it
  also holds for functions of the form
  $f=\sum_{k=1}^{m}{a_{k}\mathbbm{1}_{A_{k}}}$, where
  $0 \leq a_{k} < \infty$, $A_{k} \in \cA \otimes \cF$ for all $k$ and
  $m \in \mN$.  We call functions of this form simple functions.  For
  any non-negative $\cA \otimes \cF$-measurable function $f$, consider
  the sequence of non-negative simple functions
  $f_{n}=\min\left\{\frac{\lfloor 2^{n}f\rfloor}{2^{n}},n\right\}$.
  Then~\eqref{thpf} holds for $f_{n}$ and $f_{n} \uparrow f$.  By
  monotone convergence, $\mE^{*}(f_{n}) \uparrow \mE^{*}(f)$ and, for
  each $\sigma \in \cS$,
  \begin{equation*}
    \int_{\Omega} f_{n}(\omega, \sigma) \mP_\sigma(d\omega) \uparrow 
    \int_{\Omega} f(\omega, \sigma) \mP_\sigma(d\omega),
  \end{equation*}
  and hence
  \begin{equation*}
    \int_{\cS}\int_{\Omega} f_{n}(\omega, \sigma) \mP_\sigma(d\omega)\pi(d\sigma) \uparrow 
    \int_{\cS}\int_{\Omega} f(\omega, \sigma) \mP_\sigma(d\omega)\pi(d\sigma).
  \end{equation*}
  Hence~\eqref{thpf} holds for $f$.
\end{proof}

In the previous theorems, we have rigorously constructed a new
probability measure $\mP^{*}$ on the measurable space
$(\Omega \times \cS, \cA \otimes \cF)$ and also obtained an expression
for the integral of a $\cA \otimes \cF$-measurable function under
$\mP^*$.  We now use this expression to justify the analogy between
\ecis\ and \scis\ in the case of a discrete regime space.

\subsection{Discrete regime space}
\label{sec:tcoadrs}

We now suppose that $\cS$ is discrete, and take $\cF$ to comprise all
subsets of $\cS$.  In particular, every decision variable is
$\cF$-measurable.  Morover in this case we can, and shall, require
$\pi(\{\sigma\})>0$ for all $\sigma\in\cS$.

We will keep the notation introduced above and for a stochastic
variable $X:(\Omega, \cA) \rightarrow (\mR, \cB_{X})$ we will denote
by
$X^{*}:(\Omega \times \cS,\cA \otimes \cF) \rightarrow (\mR, \cB_{X})$
the function defined by $X^{*}(\omega,\sigma)=X(\omega)$.  Similarly
for a decision variable $\Theta: \cS \rightarrow \Theta(\cS)$ we will
denote by
$\Theta^{*}:(\Omega \times \cS,\cA \otimes \cF) \rightarrow (\mR,
\cB_{X})$ the function defined by
$\Theta^{*}(\omega,\sigma)=\Theta(\sigma)$.  We will use similar
conventions for all the variables we consider.

Now~\eqref{thpf} becomes:
\begin{align*}
  \nonumber
  \mE^{*}(f)&=\sum_{\sigma \in \cS} \int_{\Omega} f(\omega,\sigma) \mP_\sigma(d\omega) \pi(\sigma) \\
  \nonumber &= \sum_{\sigma \in \cS} \mE_\sigma(f_\sigma) \pi(\sigma).
\end{align*}

\begin{remark}
  Note that for any $X^{*}$ as above,
  $\sigma(X^{*})=\sigma(X) \times \{\cS\}$.  Similarly, for any
  $\Theta^{*}$ as above,
  $\sigma(\Theta^{*})=\{\Omega\} \times \sigma(\Theta)$.  Thus
  \begin{align*}
    \sigma(X^{*},\Theta^{*})&= \sigma(\{A_{X^{*}} \cap A_{\Theta^{*}}:
                              A_{X^{*}} \in \sigma(X^{*}), A_{\Theta^{*}} \in
                              \sigma(\Theta^{*})\})
                              \,\, \text{(see~\citet[p.~73]{res:14})} \\
                            &= \sigma(\{A_{X^{*}} \cap A_{\Theta^{*}}: A_{X^{*}} \in \sigma(X) \times \{\cS\}, A_{\Theta^{*}} \in \{\Omega\} \times                                \sigma(\Theta)\}) \\
                            &=\sigma(\{ (A_{X} \times \cS) \cap (\Omega \times A_{\Theta}): A_{X} \in \sigma(X), A_{\Theta} \in \sigma(\Theta)\}) \\
                            &=\sigma(\{ A_{X} \times A_{\Theta}: A_{X} \in \sigma(X), A_{\Theta} \in \sigma(\Theta)\}) \\
                            &=:\sigma(X) \otimes \sigma(\Theta).
  \end{align*}
  Thus, for any $\sigma \in \cS$ and $A_{X^*} \in \sigma(X^*)$, the
  function
  $\mathbbm{1}_{A_{X^*}}^{\sigma}: \Omega \rightarrow \{0,1\}$ defined
  by
  $\mathbbm{1}_{A_{X^*}}^{\sigma}(\omega):=\mathbbm{1}_{A_{X^*}}(\omega,
  \sigma)$ does not depend on $\sigma$.  It is equal to
  $\mathbbm{1}_{A_{X}}$, for $A_{X} \in \sigma(X)$ such that
  $A_{X^*}=A_{X} \times \{\cS\}$.  Also for
  $A_{X^*,\Theta^*} \in \sigma(X^*,\Theta^*)$, the function
  $\mathbbm{1}_{A_{X^*,\Theta^*}}$ is
  $(\sigma(X) \otimes \sigma(\Theta))$-measurable, and by
  \lemref{product1}, for $\sigma \in \cS$, the function
  $\mathbbm{1}_{A_{X^*,\Theta^*}}^{\sigma}: \Omega \rightarrow
  \{0,1\}$ defined by
  $\mathbbm{1}_{A_{X^*,\Theta^*}}^{\sigma}(\omega):=\mathbbm{1}_{A_{X^*,\Theta^*}}(\omega,
  \sigma)$ is $\sigma(X)$-measurable.
  $\mathbbm{1}_{A_{X^*,\Theta^*}}^{\sigma}(\omega)$ is equal to
  $\mathbbm{1}_{A_{X}^{\sigma}}$ for $A_{X}^{\sigma} \in \sigma(X)$
  such that $A_{X}^{\sigma}$ is the section of $A_{X^*,\Theta^*}$ at
  $\sigma$.
\end{remark}

\begin{theorem}
  \label{thm:equivalence}
  Let $X,Y,Z$ be $\cA$-measurable functions on $\Omega$, and let
  $\Phi,\Theta$ be decision variables on $\cS$, where $\cS$ is
  discrete.  Suppose that $\Theta$ is complementary to $\Phi$.  Also,
  let $X^*$, $Y^*$, $Z^*$ and $\Phi^*$, $\Theta^*$ be the
  corresponding $\cA \otimes \cF$-measurable functions.  Then
  \ind{X}{(Y,\Theta)}{(Z,\Phi)} if and only if
  \sind{X^*}{(Y^*,\Theta^*)}{(Z^*,\Phi^*)}.
\end{theorem}

\begin{proof}\quad\\
  $\Rightarrow$.  Since \ind{X}{(Y,\Theta)}{(Z,\Phi)}, by
  \proporef{nsprop1}, for all $\phi \in \Phi(\cS)$ and all
  $A_{X} \in \sigma(X)$, there exists a function $w_{\phi}(Z)$ such
  that, for all $\sigma \in \Phi^{-1}(\phi)$,
  \begin{equation}
    \label{eq:equivalence1a}
    \mE_\sigma(\mathbbm{1}_{A_{X}}\mid Y,Z)=w_{\phi}(Z) \,\, \text{\as{\mP_\sigma}},
  \end{equation}
  \ie,
  \begin{equation}
    \label{eq:equivalence1b}
    \mE_\sigma(\mathbbm{1}_{A_{X}}\mathbbm{1}_{A_{Y,Z}})=\mE_\sigma\left\{w_{\phi}(Z)\mathbbm{1}_{A_{Y,Z}}\right\}
    \,\, \text{whenever $A_{Y,Z} \in \sigma(Y,Z)$}.
  \end{equation}
  To show that \ind{X^*}{(Y^*,\Theta^*)}{(Z^*,\Phi^*)}, by
  \proporef{stocci} we need to show that, for all
  $A_{X^{*}} \in \sigma(X^*)$, there exists a function $w(Z^*,\Phi^*)$
  such that
  \begin{equation*}
    \label{eq:equivalence2a}
    \mE^*\left(\mathbbm{1}_{A_{X^{*}}} \mid Y^{*},\Theta^*,Z^{*},\Phi^*\right)
    =w(Z^*,\Phi^*) \,\, \text{\rm{a.s.}},
  \end{equation*}
  \ie,
  \begin{equation}
    \label{eq:equivalence2b}
    \mE^*(\mathbbm{1}_{A_{X^*}}\mathbbm{1}_{A_{Y^*,\Theta^*,Z^*,\Phi^*}})
    =\mE^*\left\{w(Z^*, \Phi^*)\mathbbm{1}_{A_{Y^*,\Theta^*,Z^*,\Phi^*}}\right\}
  \end{equation}
  whenever
  $A_{Y^*,\Theta^*,Z^*,\Phi^*} \in \sigma(Y^*,\Theta^*,Z^*,\Phi^*)$.

  Let $A_{X^{*}} \in \sigma(X^*)$ and define
  $w(z^{*},\phi^{*})=w_{\phi^{*}}(z^{*})$ as in~\eqref{equivalence1b}.
  Then for all
  $A_{Y^*,\Theta^*,Z^*,\Phi^*} \in \sigma(Y^*,\Theta^*,Z^*,\Phi^*)$,
  \begin{align*}
    \mE^*(\mathbbm{1}_{A_{X^*}}\mathbbm{1}_{A_{Y^*,\Theta^*,Z^*,\Phi^*}})
    &=\sum_{\sigma \in \cS}{\mE_\sigma(\mathbbm{1}_{A_{X}}\mathbbm{1}_{A_{Y,Z}^{\sigma}})\pi(\sigma)}\\
    &=\sum_{\sigma \in \cS}{\mE_\sigma\left\{w_{\phi}(Z)\mathbbm{1}_{A_{Y,Z}^{\sigma}}\right\} \pi(\sigma)} \,\, 
      \text{by~\eqref{equivalence1b}}\\
    &=\mE^*\left\{w(Z^*,
      \Phi^*)\mathbbm{1}_{A_{Y^*,\Theta^*,Z^*,\Phi^*}}\right\}
  \end{align*}
  which proves~\eqref{equivalence2b}.
  
  \noindent $\Leftarrow$.  To show that \ind{X}{(Y,\Theta)}{(Z,\Phi)},
  let $\phi \in \Phi(\cS)$ and $A_{X} \in \sigma(X)$.  Then, for any
  $\sigma_{o} \in \Phi^{-1}(\phi)$,
  \begin{align*}
    \mE_{\sigma_{o}}(\mathbbm{1}_{A_{X}}\mathbbm{1}_{A_{Y,Z}})\pi(\sigma_{o})
    &=\sum_{\sigma \in
      \cS}{\mE_\sigma\left\{\mathbbm{1}_{A_{X}}\mathbbm{1}_{A_{Y,Z}}\mathbbm{1}_{\sigma_o}(\sigma)\right\}
      \pi(\sigma)} \\
    &=\mE^*(\mathbbm{1}_{A_{X^*}}\mathbbm{1}_{A_{Y,Z}\times\{\sigma_o\}}) \\
    &=\mE^*\left\{w(Z^*,\Phi^*)\mathbbm{1}_{A_{Y,Z}\times\{\sigma_o\}}\right\} \,\, \text{by~\eqref{equivalence2b}}\\
    &=\sum_{\sigma \in
      \cS}{\mE_\sigma\left\{w(Z;\Phi(\sigma))\mathbbm{1}_{A_{Y,Z}}\mathbbm{1}_{\sigma_o}(\sigma)\right\}
      \pi(\sigma)} \\
    &=\mE_{\sigma_o}\left\{w(Z;\Phi(\sigma_o))\mathbbm{1}_{A_{Y,Z}}\right\}\pi(\sigma_o).
  \end{align*}
  Since $\pi(\sigma_o)>0$, we have proved~\eqref{equivalence1b} with
  $w_{\phi}(z)=w(z,\phi)$.
\end{proof}

\begin{corollary}
  \label{cor:nstoc1}
  Suppose we are given a collection of \ecis\ properties as in the
  form of \defref{nstocdef1}.  If the regime space $\cS$ is discrete,
  any deduction made using the axioms of conditional independence will
  be valid, so long as, in both premisses and conclusions, no
  non-stochastic variables appear in the left-most term in a
  conditional independence statement (we are however allowed to
  violate this condition in intermediate steps of an argument).
\end{corollary}

\begin{corollary}
  \label{cor:nstoc2}
  In the case of a discrete regime space, we have:\\[1ex]
  \noindent P$4''$: \ind{X}{(Y,\Theta)}{(Z,\Phi)} and $W \preceq X$
  $\Rightarrow$ \ind{X}{(Y,\Theta)}{(Z,W,\Phi)}.
\end{corollary}

\section{Other approaches}
\label{sec:nodiscrete}

Inspecting the proof of \thmref{equivalence}, we see that the
assumption of discreteness of the regime space $\cS$ is crucial.  If
we have an uncountable regime space $\cS$ and assign a distribution
$\pi$ over it, the arguments for the forward direction will still
apply but the arguments for the reverse direction will not.
Intuitively, this is because~\eqref{equivalence2a} holds almost
everywhere but not necessarily everywhere.  Thus we cannot immediately
extend it to hold for all $\sigma \in \cS$ as
in~\eqref{equivalence1a}.  However, using another, more direct,
approach we can still deduce P$4''$ if we impose appropriate
conditions.  In particular this will hold if the stochastic variables
are discrete.  Alternatively, we can use a domination condition on the
set of regimes.

\subsection{Discrete variables}
\label{sec:discar}
\begin{proposition}
  \label{propo:nstocaxiom4disc}
  Let $X$, $Y$, $Z$, $W$ be discrete stochastic variables, $\Phi$,
  $\Theta$ be decision variables.
  Then\\[1ex]
  \noindent P$4''$: \ind{X}{(Y,\Theta)}{(Z,\Phi)} and $W \preceq X$
  $\Rightarrow$ \ind{X}{(Y,\Theta)}{(Z,W,\Phi)}.
\end{proposition}

\begin{proof}
  To show that \ind{X}{(Y,\Theta)}{(Z,W,\Phi)} we need to show that,
  for all $\phi \in \Phi(\cS)$ and all $A_{X} \in \sigma(X)$, there
  exists $w_{\phi}(Z,W)$ such that, for all
  $\sigma \in \Phi^{-1}(\phi)$,
  \begin{equation*}
    \mE_\sigma\left(\mathbbm{1}_{A_{X}}\mid Y,Z,W\right)=w_{\phi}(Z,W) \,\, \text{\as{\mP_\sigma}},
  \end{equation*}
  \ie,
  \begin{equation}
    \label{eq:nstocaxiom4bii}
    \mE_\sigma\left(\mathbbm{1}_{A_{X}}\mathbbm{1}_{A_{Y,Z,W}}\right)
    =\mE_\sigma\left\{w_{\phi}(Z,W) \mathbbm{1}_{A_{Y,Z,W}}\right\}
    \,\, \text{whenever $A_{Y,Z,W} \in \sigma(Y,Z,W)$}.
  \end{equation}
  Observe that it is enough to show~\eqref{nstocaxiom4bii} for
  $A_{Y,Z,W} \in \sigma(Y,Z,W)$ such that $\mP_\sigma(A_{Y,Z,W}) > 0$.
  Also since $X$, $Y$, $Z$ and $W$ are discrete we need to
  show~\eqref{nstocaxiom4bii} only for sets of the form $\{X=x\}$ and
  $\{Y=y,Z=z,W=w\}$.
  Thus it is enough to show that, for all $\phi \in \Phi(\cS)$ and all
  $x$, there exists $w_{\phi}(Z,W)$ such that, for all
  $\sigma \in \Phi^{-1}(\phi)$, and all $y,z,w$ such that
  $\mP_\sigma({Y=y,Z=z,W=w}) > 0$,
  \begin{equation*}
    \mE_\sigma\left(\mathbbm{1}_{\{X=x\}}\mathbbm{1}_{\{Y=y,Z=z,W=w\}}\right)
    =\mE_\sigma\left\{w_{\phi}(Z,W)\mathbbm{1}_{\{Y=y,Z=z,W=w\}}\right\}.
  \end{equation*}

  Let $\phi \in \Phi(\cS)$.  For $\sigma \in \Phi^{-1}(\phi)$ and all
  $x,y,z,w$ such that $\mP_\sigma(Y=y,Z=z,W=w) > 0$,
  \begin{align}
    \mE_\sigma\left(\mathbbm{1}_{\{X=x\}}\mathbbm{1}_{\{Y=y,Z=z,W=w\}}\right)
    &=\mP_\sigma\left(X=x, Y=y, Z=z, W=w\right) \nonumber \\
    &=\mP_\sigma\left(X=x \mid Y=y, Z=z, W=w\right)\mP_\sigma\left(Y=y, Z=z, W=w\right) \nonumber \\
    &=\frac{\mP_\sigma\left(X=x, W=w \mid Y=y, Z=z\right)}
      {\mP_\sigma\left(W=w \mid Y=y, Z=z\right)} \,\mP_\sigma\left(Y=y,
      Z=z, W=w\right).
      \label{eq:nstocaxiom4biv}
  \end{align}
  Since \ind{X}{(Y,\Theta)}{(Z,\Phi)} and $W \preceq X$, there exist
  $w_{\phi}^{1}(Z)$ and $w_{\phi}^{2}(Z)$ such that
  \begin{equation*}
    \mE_\sigma\left(\mathbbm{1}_{\{X=x, W=w\}}\mid Y,Z\right)=w_{\phi}^{1}(Z) \,\, \text{\as{\mP_\sigma}}
  \end{equation*}
  and
  \begin{equation*}
    \mE_\sigma\left(\mathbbm{1}_{\{W=w\}}\mid Y,Z\right)=w_{\phi}^{2}(Z) \,\, \text{\as{\mP_\sigma}}
  \end{equation*} 
  where $w_{\phi}^{1}(Z)=0$ unless $w=W(x)$.

  Define
  \begin{equation*}
    w_{\phi}(z)= \left\{
      \begin{array}{rl}
        \frac{w_{\phi}^{1}(z)}{w_{\phi}^{2}(z)} & \text{if } w_{\phi}^{2}(z) \neq 0,\\
        0 & \text{if } w_{\phi}^{2}(z)=0
      \end{array} \right.
  \end{equation*}
  and note that $w_{\phi}^{2}(z) \neq 0$ when
  $\mP_\sigma(Y=y,Z=z,W=w) \neq 0$.  Also note that since
  $w_{\phi}(Z)$ is $\sigma(Z)$-measurable it is also
  $\sigma(W,Z)$-measurable.  Returning to!\eqref{nstocaxiom4biv} we
  get
  \begin{align*}
    \mE_\sigma\left(\mathbbm{1}_{\{X=x\}}\mathbbm{1}_{\{Y=y,Z=z,W=w\}}\right)
    &=w_{\phi}(z)\mP_\sigma\left(Y=y, Z=z, W=w\right) \\
    &=\mE_\sigma\left\{w_{\phi}(Z)\mathbbm{1}_{\{Y=y,Z=z,W=w\}}\right\}
  \end{align*}
  which concludes the proof.
\end{proof}

\subsection{Dominating regime}

\begin{definition}[Dominating regime]
  Let $\cS$ index a set of probability measures on $(\Omega,\cA)$.
  For $\cS_0 \subseteq \cS$, we say that $\sigma^{*} \in \cS_0$ is a
  \emph{dominating regime in $\cS_0$}, if, for all $\sigma \in \cS_0$,
  $\mP_\sigma \ll \mP_{\sigma^{*}}$; that is,
  \begin{equation*}
    \mP_{\sigma^*}(A)=0 \Rightarrow \mP_{\sigma}(A)=0 
    \,\, \text{for all $A \in \cA$ and all $\sigma\in\cS_0$.}
  \end{equation*} 
\end{definition}

\begin{proposition}
  \label{propo:nstocaxiom4}
  Let $X$, $Y$, $Z$, $W$ be stochastic variables, $\Phi$, $\Theta$ be
  decision variables.  Suppose that, for all $\phi \in \Phi(\cS)$,
  there exists a dominating regime
  $\sigma_{\phi} \in \Phi^{-1}(\phi)$.
  Then\\[1ex]
  \noindent P$4''$: \ind{X}{(Y,\Theta)}{(Z,\Phi)} and $W \preceq X$
  $\Rightarrow$ \ind{X}{(Y,\Theta)}{(Z,W,\Phi)}.
\end{proposition}

\begin{proof}
  By \proporef{defbreak}, it suffices to prove the following two
  statements:
  \begin{equation}
    \label{eq:nstocaxioms4a} 
    \ind{X}{Y}{(Z,W,\Phi,\Theta)} 
  \end{equation}
  and
  \begin{equation}
    \label{eq:nstocaxioms4b} 
    \ind{X}{\Theta}{(Z,W,\Phi)}.  
  \end{equation}
  To prove~\eqref{nstocaxioms4a}, we will use \proporef{nstocsym} and
  prove equivalently that \mbox{$Y \cip X \mid $} $(Z,W,\Phi,\Theta)$.
  Note first that since \ind{X}{(Y,\Theta)}{(Z,\Phi)}, by
  \proporef{defbreak} it follows that \ind{X}{Y}{(Z,\Phi,\Theta)}, and
  by \proporef{nstocsym} it follows that ${Y} \cip {X} \mid$
  ${(Z,\Phi,\Theta)}$.  Also, since $W \preceq X$, by
  \lemref{measurability} it follows that
  $\sigma(W) \subseteq \sigma(X)$.  Let
  $(\phi,\theta) \in (\Phi(\cS),\Theta(\cS))$,
  $\sigma=(\Phi,\Theta)^{-1}(\phi,\theta)$ and $A_{Y} \in \sigma(Y)$.
  Then
  \begin{align*}
    \mE_\sigma\left(\mathbbm{1}_{A_{Y}}\mid X,Z,W\right) &=\mE_\sigma\left(\mathbbm{1}_{A_{Y}}\mid X,Z\right) \,\, \text{\as{\mP_\sigma} since $\sigma(W) \subseteq \sigma(X)$} \\
                                                         &=\mE_\sigma\left(\mathbbm{1}_{A_{Y}}\mid Z\right) \,\,
                                                           \text{\as{\mP_\sigma} since \ind{Y}{X}{(Z,\Phi,\Theta)}},
  \end{align*}
  which proves that \ind{Y}{X}{(Z,W,\Phi,\Theta)}.

  To prove~\eqref{nstocaxioms4b}, let $\phi \in \Phi(\cS)$ and
  $A_{X} \in \sigma(X)$.  We will show that there exists
  $w_{\phi}(Z,W)$ such that, for all $\sigma \in \Phi^{-1}(\phi)$,
  \begin{equation*}
    \mE_\sigma\left(\mathbbm{1}_{A_{X}}\mid Z,W\right) = w_{\phi}(Z,W) \,\, \text{\as{\mP_\sigma}},
  \end{equation*}
  \ie,
  \begin{equation*}
    \mE_\sigma\left(\mathbbm{1}_{A_{X}}\mathbbm{1}_{A_{Z,W}}\right)
    =\mE_\sigma\left\{w_{\phi}(Z,W)\mathbbm{1}_{A_{Z,W}}\right\} 
    \,\, \text{whenever $A_{Z,W} \in \sigma(Z,W)$}.
  \end{equation*}
  Let $A_{Z,W} \in \sigma(Z,W)$ and note that
  \begin{equation}
    \label{eq:nstocaxiom4aiii} 
    \mE_\sigma\left(\mathbbm{1}_{A_{X}}\mathbbm{1}_{A_{Z,W}}\right)
    =\mE_\sigma\left\{\mE_\sigma(\mathbbm{1}_{A_{X}}\mathbbm{1}_{A_{Z,W}}\mid Z)\right\}.  
  \end{equation} 
  Since \ind{X}{(Y,\Theta)}{(Z,\Phi)}, by \lemref{nstocaxiom3i} it
  follows that \ind{(X,Z)}{(Y,\Theta)}{(Z,\Phi)}, and by
  \proporef{defbreak} that \ind{(X,Z)}{\Theta}{(Z,\Phi)}.  Also, since
  $W \preceq X$ there exists $a_{\phi}(Z)$ such that
  \begin{equation}
    \label{eq:nstocaxiom4aiv}
    \mE_\sigma(\mathbbm{1}_{A_{X}} \mathbbm{1}_{A_{Z,W}}\mid Z)=a_{\phi}(Z) \,\, \text{\as{\mP_\sigma}}.
  \end{equation}
  In particular, for the dominating regime
  $\sigma_{\phi} \in \Phi^{-1}(\phi)$,
  \begin{equation*}
    \mE_{\sigma_{\phi}}(\mathbbm{1}_{A_{X}} \mathbbm{1}_{A_{Z,W}}\mid Z)=a_{\phi}(Z) \,\, \text{\as{\mP_{\sigma_{\phi}}}}
  \end{equation*}
  and thus
  \begin{equation*}
    \mE_{\sigma_{\phi}}\left\{\mE_{\sigma_{\phi}}(\mathbbm{1}_{A_{X}} \mathbbm{1}_{A_{Z,W}}\mid Z,W)\mid Z\right\}
    =a_{\phi}(Z) \,\, \text{\as{\mP_{\sigma_{\phi}}}}.
  \end{equation*}
  Since $\mP_\sigma \ll \mP_{\sigma_{\phi}}$ for all
  $\sigma \in \Phi^{-1}(\phi)$, it follows that, for all
  $\sigma \in \Phi^{-1}(\phi)$,
  \begin{equation}
    \label{eq:nstocaxiom4av}
    \mE_{\sigma_{\phi}}\left\{\mE_{\sigma_{\phi}}(\mathbbm{1}_{A_{X}} \mathbbm{1}_{A_{Z,W}} \mid Z,W)\mid Z\right\}
    =a_{\phi}(Z) \,\, \text{\as{\mP_\sigma}}.
  \end{equation}
  Thus, by~\eqref{nstocaxiom4aiv} and~\eqref{nstocaxiom4av}, we get
  that
  \begin{equation}
    \label{eq:nstocaxiom4avi}
    \mE_\sigma(\mathbbm{1}_{A_{X}} \mathbbm{1}_{A_{Z,W}}\mid Z)= 
    \mE_{\sigma_{\phi}}\left\{\mE_{\sigma_{\phi}}(\mathbbm{1}_{A_{X}} \mathbbm{1}_{A_{Z,W}}\mid Z,W)\mid Z\right\}
    \,\, \text{\as{\mP_\sigma}}.
  \end{equation}
  Similarly,
  \begin{equation}
    \label{eq:nstocaxiom4avii}
    \mE_{\sigma_{\phi}}\left\{\mE_{\sigma_{\phi}}(\mathbbm{1}_{A_{X}} \mathbbm{1}_{A_{Z,W}}\mid Z,W)\mid Z\right\}=
    \mE_\sigma\left\{\mE_{\sigma_{\phi}}(\mathbbm{1}_{A_{X}} \mathbbm{1}_{A_{Z,W}}\mid Z,W)\mid Z\right\}
    \,\, \text{\as{\mP_\sigma}}.
  \end{equation}
  Returning to~\eqref{nstocaxiom4aiii}, it follows that
  \begin{align*}
    \mE_\sigma\left(\mathbbm{1}_{A_{X}}\mathbbm{1}_{A_{Z,W}}\right)
    &=\mE_\sigma\left[\mE_{\sigma_{\phi}}\left\{\mathbbm{1}_{A_{Z,W}}\mE_{\sigma_{\phi}}(\mathbbm{1}_{A_{X}}\mid
      Z,W) \mid Z\right\}\right]
      \,\, \text{by~\eqref{nstocaxiom4avi}}\\
    &=\mE_\sigma\left[\mE_\sigma\left\{\mathbbm{1}_{A_{Z,W}}\mE_{\sigma_{\phi}}(\mathbbm{1}_{A_{X}}\mid Z,W) \mid Z\right\}\right]
      \,\, \text{by~\eqref{nstocaxiom4avii}}\\
    &=\mE_\sigma\left\{\mathbbm{1}_{A_{Z,W}}\mE_{\sigma_{\phi}}(\mathbbm{1}_{A_{X}}\mid
      Z,W)\right\}.
  \end{align*}
\end{proof}

\section{Pairwise conditional independence}
\label{sec:pairwise}

Yet another path is to relax the notion of \ecis.  Here we introduce a
weaker version that we term \emph{pairwise extended conditional
  independence}.

\begin{definition}
  \label{def:nstocdef2}
  Let $X$, $Y$ and $Z$ be stochastic variables and let $\Theta$ and
  $\Phi$ be decision variables.  We say that \emph{$X$ is pairwise
    (conditionally) independent of $(Y,\Theta)$ given $(Z,\Phi)$}, and
  write \pind{X}{(Y,\Theta)}{(Z,\Phi)}, if for all
  $\phi \in \Phi(\cS)$, all real, bounded and measurable functions
  $h$, and all pairs $\{\sigma_{1},\sigma_{2}\} \in \Phi^{-1}(\phi)$,
  there exists a function $w_{\phi}^{\sigma_{1},\sigma_{2}}(Z)$ such
  that
  \begin{equation*}
    \mE_{\sigma_{1}}\left\{h(X)\mid Y,Z\right\}=w_{\phi}^{\sigma_{1},\sigma_{2}}(Z) \,\, \text{\as{\mP_{\sigma_{1}}}}
  \end{equation*}
  and
  \begin{equation*}
    \mE_{\sigma_{2}}\left\{h(X)\mid Y,Z\right\}=w_{\phi}^{\sigma_{1},\sigma_{2}}(Z) \,\, \text{\as{\mP_{\sigma_{2}}}}.
  \end{equation*}
\end{definition}

It is readily seen that \ecis\ implies \pecis, but the converse is
false.  In \defref{nstocdef2}, for all $\phi \in \Phi(\cS)$, we only
require a common version for the corresponding conditional expectation
for every pair of regimes
$\{\sigma_{1},\sigma_{2}\}\in \Phi^{-1}(\phi)$, but we do not require
that these versions agree on one function that can serve as a version
for the corresponding conditional expectation simultaneously in all
regimes $\sigma \in \Phi^{-1}(\phi)$.

Under this weaker definition, the analogues of P$1'$ to P$5'$, and of
P$3''$ and P$5''$, can be seen to hold just as in \secref{asa}.  Also,
by confining attention to two regimes at a time and applying
\corref{nstoc2}, the analogue of P$4''$ will hold without further
conditions.

It can be shown that, when there exists a dominating regime, \pecis\
is equivalent to \ecis.  This property can be used to supply an
alternative proof of \proporef{nstocaxiom4}.

\section{Further extensions}
\label{sec:tcodrv}

So far we have studied extended conditional independence relations of
the form \ind{X}{(Y,\Theta)}{(Z,\Phi)}, where the left-most term is
fully stochastic.  We now wish to extend this to the most general
expression, of the form \ind{(X, K)}{(Y,\Theta)}{(Z,\Phi)}, where
$X, Y, Z$ are stochastic variables and $K, \Theta, \Phi$ are decision
variables, and to investigate the validity of the separoid axioms.

Consider first the expression \ind{K}{\Theta}{Z}.  Our desired
intuitive interpretation of this is that conditioning on the
stochastic variable $Z$ renders the decision variables $K$ and
$\Theta$ variation independent.  We need to turn this intuition into a
rigorous definition, taking account of the fact that $Z$ may have
different distributions in the different regimes $\sigma \in \cS$,
whereas $K$ and $\Theta$ are functions defined on $\cS$.


One way to interpret this intuition is to consider, for each value $z$
of $Z$, the set $ \cS_z$ of regimes that for which $z$ is a ``possible
outcome,'' and ask that the decision variables be variation
independent on this restricted set.  In order to make this rigorous,
we shall require that $Z: (\Omega, \cA) \rightarrow (F_Z, {\cal F}_Z)$
where $(F_Z, {\cal F}_Z)$ is a topological space with its Borel
$\sigma$-algebra, and introduce
\begin{equation*}
  \cS_z:=\left\{\sigma \in \cS: \mP_\sigma(Z\in U) > 0\mbox{ for every open set }
    U\subseteq {\cal F}_Z \mbox{ containing }z\right\}.
\end{equation*} 

In particular, when $Z$ is discrete, with the discrete topology,
\begin{equation*}
  \cS_z:=\left\{\sigma \in \cS: \mP_\sigma(Z=z) > 0\right\}.
\end{equation*} 

We now formalise the slightly more general expression
\ind{K}{\Theta}{(Z,\Phi)} in the following definition.

 
\begin{definition}
  \label{def:nstocdiscrete1}
  Let $Z$ be a stochastic variable and $K, \Theta, \Phi$ decision
  variables.  We say that \emph{$\Theta$ is (conditionally)
    independent of $K$ given $(Z,\Phi)$}, and write
  \ind{\Theta}{K}{(Z,\Phi)} if, for all $z \in Z(\Omega)$,
  \vind{\Theta}{K}{\Phi} $[\cS_z]$.
\end{definition}


We now wish to introduce further definitions, to allow stochastic and
decision variables to appear together in the left-most term of a
conditional independence statement.  Recall that in
\defref{nstocdef1}, we defined $\ind{X}{(Y,\Theta)}{(Z,\Phi)}$ only
when $\Theta$ and $\Phi$ are complementary on $\cS$.  Similarly, we
will define \ind{(X,K)}{(Y,\Theta)}{(Z,\Phi)} only when $K$, $\Theta$
and $\Phi$ are complementary, \ie, when the function $(K,\Theta,\Phi)$
on $\cS$ is a surjection.  Our interpretation of
$\ind{(X,K)}{(Y,\Theta)}{(Z,\Phi)}$ will now be a combination of
Definitions~\ref{def:nstocdef1} and~\ref{def:nstocdiscrete1}.  We
start with a special case.

\begin{definition}
  Let $Y, Z$ be stochastic variables, and $K, \Theta, \Phi$
  complementary decision variables.  We say that \emph{$(Y,\Theta)$ is
    (conditionally) independent of $K$ given $(Z,\Phi)$}, and write
  \ind{(Y,\Theta)}{K}{(Z,\Phi)}, if \ind{Y}{K}{(Z,\Phi,\Theta)} and
  \ind{\Theta}{K}{(Z,\Phi)}.  In this case we may also say that
  \emph{$K$ is (conditionally) independent of $(Y,\Theta)$ given
    $(Z,\Phi)$}, and write \ind{K}{(Y,\Theta)}{(Z,\Phi)}.
\end{definition}

Finally, the general definition:
\begin{definition}
  \label{def:full}
  Let $X,Y,Z$ be stochastic variables and $K, \Theta, \Phi$
  complementary decision variables.  We say that \emph{$(X,K)$ is
    (conditionally) independent of $(Y,\Theta)$ given $(Z,\Phi)$}, and
  write \ind{(X,K)}{(Y,\Theta)}{(Z,\Phi)}, if
  $$\ind{X}{(Y,\Theta)}{(Z,\Phi,K)},\,\mbox{ and }\,\ind{K}{(Y,\Theta)}{(Z,\Phi)}.$$
\end{definition}

\subsection{Separoid properties}
\label{sec:gensep}

We now wish to investigate the extent to which versions of the
separoid axioms apply to the above general definition.  In this
context the relevant set $V$ is the set of pairs of the form
$(Y,\Theta)$, where $Y$ is a stochastic variable defined on $\Omega$
and $\Theta$ is a decision variable defined on $\cS$.

For a full separoid treatment we also need to introduce a quasiorder
$\preceq$ on $V$.  A natural definition would be:
$(W,\Lambda) \preceq (Y,\Theta)$ if $W =f(Y)$ for some measurable
function $f$ (also denoted by $W \preceq Y$) and $\Lambda=h(\Theta)$
for some function $h$.
Then $(V, \preceq)$ becomes a join semilattice, with join
$(Y,\Theta)\vee(W,\Lambda)\approx ((Y,W),(\Theta,\Lambda))$.  However,
whenever we consider a relation \ind{(X,K)}{(Y,\Theta)}{(Z,\Phi)} we
require that $K,\Theta,\Phi$ be complementary, a property that would
typically be lost on replacing $\Theta$ by a non-trivial reduction
$\Lambda = h(\Theta)$.  Consequently, we will only be considering
non-trivial reductions of stochastic variables.  This constraint
modifies and reduces the separoid properties.

\begin{theorem}[Separoid-type properties]
  \label{axiomsgeneral}
  Let $X,Y,Z,W$ be stochastic variables and $K,\Theta,\Phi$ be
  decision variables.  Then the following properties hold:
  \begin{itemize}
  \item [P$1^g$.] \ind{(X,K)}{(Y,\Theta)}{(Z,\Phi)} $\Rightarrow$
    \ind{(Y,\Theta)}{(X,K)}{(Z,\Phi)}.
  \item [P$2^g$.] \ind{(X,K)}{(Y,\Theta)}{(Y,\Theta)}.
  \item [P$3^g$.] \ind{(X,K)}{(Y,\Theta)}{(Z,\Phi)}, $W \preceq Y$
    $\Rightarrow$ \ind{(X,K)}{(W,\Theta)}{(Z,\Phi)}.
  \item [P$4^g$.] Under the conditions of \corref{nstoc2},
    \proporef{nstocaxiom4disc} and
    \proporef{nstocaxiom4},\\[1ex]
    \ind{(X,K)}{(Y,\Theta)}{(Z,\Phi)}, $W \preceq Y$ $\Rightarrow$
    \ind{(X,K)}{Y}{(Z,W,\Phi,\Theta)}.
  \item [P$5^g$.]  $ \left.
      \begin{array}{c}
        \ind{(X,K)}{(Y,\Theta)}{(Z,\Phi)}\\
        \text{ and }\\
        \ind{(X,K)}{W}{(Y,Z,\Theta,\Phi)}
      \end{array} 
    \right\} \,\, \Rightarrow \,\,
    \ind{(X,K)}{(Y,W,\Theta)}{(Z,\Phi)}.  $
  \end{itemize}
\end{theorem}

\begin{proof}\quad\\
  P$1^g$.  By \defref{full} and \proporef{defbreak}, we need to show
  that
  \begin{align}
    &\ind{Y}{X}{(Z,\Phi,\Theta,K)} \label{eq:P1i}
    \\
    &\ind{Y}{K}{(Z,\Phi,\Theta)} \label{eq:P1ii}
    \\
    &\ind{X}{\Theta}{(Z,\Phi,K)} \label{eq:P1iii}
    \\
    &\ind{K}{\Theta}{(Z,\Phi)}.  \label{eq:P1iv}
  \end{align}
  Since \ind{(X,K)}{(Y,\Theta)}{(Z,\Phi)}, we have that
  \begin{align}
    &\ind{X}{Y}{(Z,\Phi,K,\Theta)} \label{eq:P1i'}
    \\
    &\ind{X}{\Theta}{(Z,\Phi,K)} \label{eq:P1ii'}
    \\
    &\ind{Y}{K}{(Z,\Phi,\Theta)} \label{eq:P1iii'}
    \\
    &\ind{\Theta}{K}{(Z,\Phi)}.  \label{eq:P1iv'}
  \end{align}
  It follows that~\eqref{P1ii} and~\eqref{P1iii} hold automatically.
  Also applying P$1'$ to~\eqref{P1i'} we deduce~\eqref{P1i}.
  Rephrasing~\eqref{P1iv'} in terms of variation independence, we have
  that, for all $z \in Z(\Omega)$, \vind{\Theta}{K}{\Phi} $[\cS_z]$.
  Thus applying P$1^{v}$ to~\eqref{P1iv'} we deduce that, for all
  $z \in Z(\Omega)$, \vind{K}{\Theta}{\Phi} $[\cS_z]$,
  \ie~\eqref{P1iv}.

  \noindent P$2^g$.  By \defref{full}, we need to show that
  \begin{align}
    & \ind{X}{(Y,\Theta)}{(Y,\Theta,K)}
      \label{eq:P2i}
    \\
    & \ind{Y}{K}{(Y,\Theta)}
      \label{eq:P2ii}
    \\
    & \ind{\Theta}{K}{(Y,\Theta)}.
      \label{eq:P2iii}
  \end{align}
  By P$2'$ we have that \ind{X}{(Y,\Theta,K)}{(Y,\Theta,K)} which is
  identical to~\eqref{P2i}.  To show~\eqref{P2ii}, let
  $\theta \in \Theta(\cS)$ and $A_{Y} \in \sigma(Y)$.  We seek
  $w_{\theta}(Y)$ such that, for all $\sigma \in \Theta^{-1}(\theta)$,
  \begin{equation*}
    \mE_\sigma\left(\mathbbm{1}_{A_{Y}} \mid Y\right)= w_{\theta}(Y)
    \,\, \text{\as{\mP_\sigma}}.
  \end{equation*}
  But note that
  \begin{equation*}
    \mE_\sigma\left(\mathbbm{1}_{A_{Y}} \mid Y\right)= 
    \mathbbm{1}_{A_{Y}} \,\, \text{\as{\mP_\sigma}}.
  \end{equation*}
  To show~\eqref{P2iii}, let $y \in Y(\Omega)$.  By P$2^{v}$, we have
  that
  \begin{equation}
    \vind{K}{\Theta}{\Theta} \,\, [\cS_y].
  \end{equation} 
  Applying P$1^{v}$ to~\eqref{P2ii} we deduce that
  \vind{\Theta}{K}{\Theta} $[\cS_y]$, \ie~\eqref{P2iii}.
	
  \noindent P$3^g$.  By \defref{full}, we need to show that
  \begin{align}
    &\ind{X}{(W,\Theta)}{(Z,\Phi,K)} \label{eq:P3i}
    \\
    &\ind{W}{K}{(Z,\Phi,\Theta)} \label{eq:P3ii}
    \\
    &\ind{\Theta}{K}{(Z,\Phi)}.  \label{eq:P3iii}
  \end{align}
  Since \ind{(X,K)}{(Y,\Theta)}{(Z,\Phi)}, we have that
  \begin{align}
    &\ind{X}{(Y,\Theta)}{(Z,\Phi,K)} \label{eq:P3i'}
    \\
    &\ind{Y}{K}{(Z,\Phi,\Theta)} \label{eq:P3ii'}
    \\
    &\ind{\Theta}{K}{(Z,\Phi)}.  \label{eq:P3iii'}
  \end{align}
  Since $W \preceq Y$, applying P$3'$ to~\eqref{P3i'}, we
  deduce~\eqref{P3i}.  Also, applying P$3''$ to~\eqref{P3ii'} we
  deduce~\eqref{P3ii}.
 
  \noindent P$4^g$.  By~\defref{full}, we need to show that
  \begin{align}
    &\ind{X}{Y}{(Z,W,\Phi,\Theta, K)} \label{eq:P4i}
    \\
    &\ind{Y}{K}{(Z,W,\Phi,\Theta)}.  \label{eq:P4ii}
  \end{align}
  Since \ind{(X,K)}{(Y,\Theta)}{(Z,\Phi)}, we have that
  \begin{align}
    &\ind{X}{(Y,\Theta)}{(Z,\Phi,K)} \label{eq:P4i'} \\
    &\ind{Y}{K}{(Z,\Phi,\Theta)} \label{eq:P4ii'} \\
    &\ind{\Theta}{K}{(Z,\Phi)} \label{eq:P4iii'}.
  \end{align}
  Since $W \preceq Y$, applying P$4'$ to~\eqref{P4i'} we deduce that
  \ind{X}{(Y,\Theta)}{(Z,W,\Phi,K)} which implies~\eqref{P4i}.  Also,
  under the additional conditions assumed,~\eqref{P4ii'}
  implies~\eqref{P4ii}.

  \noindent P$5^g$.  By \defref{full}, we need to show that
  \begin{align}
    &\ind{X}{(Y,W,\Theta)}{(Z,\Phi,K)} \label{eq:P5i} \\
    &\ind{(Y,W)}{K}{(Z,\Phi,\Theta)} \label{eq:P5ii} \\
    &\ind{\Theta}{K}{(Z,\Phi)} \label{eq:P5iii}.
  \end{align}
  Since \ind{(X,K)}{(Y,\Theta)}{(Z,\Phi)} and
  \ind{(X,K)}{W}{(Y,Z,\Theta,\Phi)}, we have that
  \begin{align}
    &\ind{X}{(Y,\Theta)}{(Z,\Phi,K)} \label{eq:P5i'}
    \\
    &\ind{Y}{K}{(Z,\Phi,\Theta)} \label{eq:P5ii'}
    \\
    &\ind{\Theta}{K}{(Z,\Phi)} \label{eq:P5iii'}
    \\
    &\ind{X}{W}{(Y,Z,\Theta,\Phi,K)} \label{eq:P5iv'}
    \\
    &\ind{W}{K}{(Y,Z,\Theta,\Phi)}.  \label{eq:P5v'}
  \end{align}
  It follows that~\eqref{P5iii} holds automatically.  Also applying
  P$5'$ to~\eqref{P5i'} and~\eqref{P5iv'} we deduce~\eqref{P5i} and
  applying P$5''$ to~\eqref{P5ii'} and~\eqref{P5v'} we
  deduce~\eqref{P5ii}.
\end{proof}

\section{Applications of \eci}
\label{sec:4}

The driving force behind this work was the need to establish a
rigorous basis for a wide range of statistical concepts---in
particular the \dts\ for statistical causality.  In the \dts\ we have
stochastic variables whose outcomes are determined by Nature, and
decision variables that determine the probabilistic regime generating
the stochastic variables.  Using the language of \eci\ we are able to
express and manipulate conditions that allow us to transfer
probabilistic information between regimes, and thus use information
gleaned from one regime to understand a different, unobserved regime
of interest.

Here we illustrate, with two examples, how the language and the
calculus of \ecis\ can be applied to identify causal quantities.  For
numerous further applications see~\citet{apd:annrev}.

\begin{ex}[Average causal effect]
  Suppose we are concerned with the effect of a binary treatment $T$
  (with value 1 denoting active treatment, and $0$ denoting placebo)
  on a disease variable $Y$.  There are 3 regimes of interest,
  indicated by a regime indicator $\Sigma$: $\Sigma= 1$ [resp.,
  $\Sigma=0$] denotes the situation where the patient is assigned
  treatment $T=1$ [resp., $T=0$] by external intervention; whereas
  $\Sigma = \emptyset$ indicates an observational regime, in which $T$
  is chosen, in some random way beyond the analyst's control,``by
  Nature.''  For example, the data may have been gathered by doctors
  or in hospitals, and the criteria on which the treatment decisions
  were based not recorded.

  A typical focus of interest is the \emph{Average Causal Effect}
  (\ace)~\citep{guodaw:10, gendaw:11},
  \begin{equation*}
    \ace := \mE_{1}(Y)-\mE_{0}(Y)
  \end{equation*}
  where $\mE_{\sigma}(\cdot) = \mE(\cdot \mid \sigma)$ denotes
  expectation under regime $\Sigma=\sigma$.  This is a direct
  comparison of the average effects of giving treatment versus placebo
  for a given patient.  However in practice, for various reasons
  (ethical, financial, pragmatic, \etc), we may not be able to observe
  $Y$ under these interventional regimes, and then can not compare
  them directly.  Instead, we might have access to data generated
  under the observational regime, where other variables might affect
  both the treatment choice and the variable of interest.  In such a
  case, the distribution of the outcome of interest, for a patient
  receiving treatment $T=t$, cannot necessarily be assumed to be the
  same as in the corresponding interventional regime $\Sigma=t$.

  However, if the observational data have been generated and collected
  from a Randomised Control Trial (\ie the sample is randomly chosen
  and the treatment is randomly allocated), we could reasonably impose
  the following \eci\ condition:
  \begin{equation}
    \label{eq:noconfounding}
    \ind{Y}{\Sigma}{T}.  
  \end{equation}
  This condition expresses the property that, given information on the
  treatment $T$, the distribution of $Y$ is independent of the
  regime---in particular, the same under interventional and
  observational conditions.  When it holds, we are, intuitively,
  justified in identifying $\mE_t(Y)$ with $\mE_\emptyset(Y \mid T=t)$
  ($t=0,1$), so allowing estimation of \ace\ from the available data.

  To make this intuition precise, note that, according to
  \defref{nstocdef1}, property~\eqref{noconfounding} implies that
  there exists $w(T)$ such that, for all
  $\sigma \in \{\emptyset,0,1\}$,
  \begin{equation*}
    \mE_\sigma(Y \mid T)= w(T) \,\, \as{\mP_\sigma}.
  \end{equation*}
  Now in the interventional regimes, for $t=0,1$, $\mP_{t}(T=t)=1$.
  Thus for $t=0,1$,
  \begin{equation*}
    w(t) =  \mE_{t}(Y \mid T=t) = \mE_{t}(Y) \,\, \as{\mP_t}.
  \end{equation*}
  Since both $w(t)$ and $\mE_{t}(Y)$ are non-random real numbers, we
  thus must have
  \begin{equation}
    \label{eq:ex30}
    w(t) =   \mE_{t}(Y).
  \end{equation}
  Also, in the observational regime,
  \begin{eqnarray*}
    \mE_\emptyset(Y \mid T) &=& w(T) \,\, \as{\mP_\emptyset}.
  \end{eqnarray*}
  Thus (so long as in the observational regime both treatments are
  allocated with positive probability) we obtain, for $t=0,1$,
  \begin{eqnarray*}
    \mE_{\emptyset}(Y \mid T=t) &=& w(t)\\
                                &=&\mE_{t}(Y) \,\, \mbox{by~\eqref{ex30}.}
  \end{eqnarray*}
  Then
  \begin{align*}
    \ace &=\mE_{1}(Y)-\mE_{0}(Y) \\
         &=\mE_{\emptyset}(Y \mid T=1)-\mE_{\emptyset}(Y \mid T=0)
  \end{align*}
  and so \ace\ can be estimated from the observational data.
\end{ex}

\begin{ex}[Dynamic treatment strategies]
  Suppose we wish to control some variable of interest through a
  sequence of consecutive actions~\citep{rob:86, rob:87, rob:89}.  An
  example in a medical context is maintaining a critical variable,
  such as blood pressure, within an appropriate risk-free range.  To
  achieve such control, the doctor will administer treatments over a
  number of stages, taking into account, at each stage, a record of
  the patient's history, which provides information on the level of
  the critical variable, and possibly other related measurements.

  We consider two sets of stochastic variables: $\cal{L}$, a set of
  \emph{observable} variables, and $\cal{A}$, a set of \emph{action}
  variables.  The variables in ${\cal L}$ represent initial or
  intermediate symptoms, reactions, personal information, \etc,
  observable between consecutive treatments, and over which we have no
  direct control; they are perceived as generated and revealed by
  Nature.  The action variables ${\cal A}$ represent the treatments,
  which we could either control by external intervention, or else
  leave to Nature (or the doctor) to determine.

  An alternating ordered sequence
  ${\cal I}:= (L_1, A_1,\ldots, L_n, A_n, L_{n+1}\equiv Y)$ with
  $L_i\subseteq{\cal L}$ and $A_i\in{\cal A}$ defines an
  \emph{information base}, the interpretation being that the specified
  variables are observed in this time order.Thus at each stage $i$
  ($=1,\ldots, n$) we will have a realisation of the random variable
  (or set of random variables) $L_i \subseteq \cal{L}$, followed by a
  value for the variable $A_i \in \cal{A}$.  After the realization of
  the final $A_n \in \cal{A}$, we will observe the outcome variable
  $L_{n+1} \in \cal{L}$, which we also denote by $Y$.

  In such a problems we might be interested to evaluate and compare
  different strategies, \ie, well-specified algorithms that take as
  input the recorded history of a patient at each stage and give as
  output the choice (possibly randomised) of the next treatment to be
  allocated.  These strategies constitute interventional regimes, for
  which we would like to make inference.  However, it may not be
  possible to implement all (or any) of these strategies to gather
  data, so we may need to rely on observational data and hope that it
  will be possible to use these data to estimate the interventional
  effects of interest.

  We thus take the regime space to be $\cS=\{\emptyset\} \cup \cS^*$,
  where $\emptyset$ labels the observational regime under which data
  have been gathered, and $\cS^*$ is the collection of contemplated
  interventional strategies.  We denote the regime indicator, taking
  values in $\cS$, by $\Sigma$.  In order to identify the effect of
  some strategy $s \in \cS^*$ on the outcome variable $Y$, we aim to
  estimate, from the observational data gathered under regime
  $\Sigma=\emptyset$, the expectation $\mE_{s}\{k(Y)\}$, for some
  appropriate function $k(\cdot)$ of $Y$, that would result from
  application of strategy $s$.

  One way to compute $\mE_{s}\{k(Y)\}$ is by identifying the overall
  joint density of $(L_1,A_1, \ldots, L_n, A_n, Y)$ in the
  interventional regime of interest $s$.  Factorising this joint
  density, we have:

  \begin{equation*}
    p_{s}(y, \overline{l}, \overline a) = \left\{\prod_{i=1}^{n+1}
      p_{s}(l_i \mid \overline{l}_{i-1},  \overline{a}_{i-1})\right\} \times
    \left\{\prod_{i=1}^n p_{s}(a_i \mid \overline{l}_{i},
      \overline{a}_{i-1})\right\}
  \end{equation*}
  with $l_{n+1} \equiv y$.  Here $\overline{l}_i$ denotes
  $(l_1,\ldots,l_i)$, \etc

  In order to compute $\mE_{s}\{k(Y)\}$, we thus need the following
  terms:

  \begin{enumerate}[label=(\roman{enumi})]
  \item
    \label{it:ai}
    $p_{s}(a_i \mid \overline l_{i}, \overline a_{i-1})$ for
    $i = 1, \ldots,n$.
  \item
    \label{it:li}
    $p_{s}(l_i \mid \overline l_{i-1}, \overline a_{i-1})$ for
    $i = 1, \ldots, n+1$.
  \end{enumerate}

  Since $s$ is an interventional regime, corresponding to a
  well-defined treatment strategy, the terms in \itref{ai} are fully
  specified by the treatment protocol.  So we only need to get a
  handle on the terms in \itref{li}.

  One assumption that would allow this is \emph{simple stability},
  expressed as
  \begin{equation*}
    \ind {L_i} {\Sigma} {(\overline L_{i-1}, \overline A_{i-1})} 
    \quad\quad (i =1, \ldots, n+1).
  \end{equation*}
  This says, intuitively, that the distribution of $L_i$, given all
  past observations, is the same in both the interventional and the
  observational regimes.  When it holds (and assuming that the
  conditioning event occurs with positive probability in the
  observational regime) we can replace
  $p_{s}(l_i \mid \overline l_{i-1}, \overline a_{i-1})$ in \itref{li}
  with its observationally estimable counterpart,
  $p_\emptyset(l_i \mid \overline l_{i-1}, \overline a_{i-1})$.  We
  then have all the ingredients needed to estimate the interventional
  effect $\mE_{s}\{k(Y)\}$.\footnote{The actual computation can be
    streamlined using {\em $G$-recursion\/}~\citep{diddaw:10}.}

  However, in many cases the presence of unmeasured variables, both
  influencing the actions taken under the observational regime and
  affecting their outcomes, would not support a direct assumption of
  simple stability.  Denote these additional variables by $U_i$
  ($i=1,\ldots,n$).  A condition that might be more justifiable in
  this context is \emph{extended stability}, expressed as
  \begin{equation*}
    \ind {(L_i, U_i)} {\Sigma} {(\overline L_{i-1}, \overline U_{i-1} \overline A_{i-1})}
    \quad\quad (i =1, \ldots, n+1).
  \end{equation*}
  This is like simple stability, but taking the unmeasured variables
  also into account.

  Now extended stability does not, in general, imply simple stability.
  But using the machinery of \ecis, we can explore when, in
  combination with further conditions that might also be
  justifiable---for example, sequential randomisation or sequential
  irrelevance~\citep{diddaw:10, dawcon:14}---simple stability can
  still be deduced, and hence $\mE_{s}\{k(Y)\}$ estimated.
\end{ex}

\section{Discussion}
\label{sec:disc}

We have presented a rigorous account of the hitherto informal concept
of \ecis, and indicated its fruitfulness in numerous statistical
contexts, such as ancillarity, sufficiency, causal inference,\etc\@

\emph{Graphical models}, in the form of Directed Acyclic Graphs
(DAGs), are often used to represent collections of conditional
independence properties amongst stochastic variables~\citep{laudaw:90,
  cow:07}, and we can then use graphical techniques (in particular,
the {\em $d$-separation\/}, or the equivalent {\em moralization\/},
criterion) to derive, in a visual and transparent way, implied
conditional independence properties that follow from the assumptions.
When such graphical models are extended to \emph{Influence Diagrams},
incorporating both stochastic and non-stochastic variables, the
identical methods support causal inference~\citep{daw:02}.  Numerous
application may be found in~\citet{apd:annrev}.  The theory developed
in this paper formally justifies this extended methodology.

\appendix

\section{Appendix}
\label{app:A}

\begin{proof0}{\lemref{measurability}}\quad\\
  \itref{aa1} to \itref{aa2}: By definition
  $\sigma(Z)=\{Z^{-1}(A): A \in \cF_Z\}$.  Observe that for all
  $z \in F_{Z}$ and all $B \in \sigma(Z)$, either
  $\{\omega \in \Omega: Z(\omega)=z\}\cap B= \emptyset$ or
  $\{\omega \in \Omega: Z(\omega)=z\} \subseteq B$.  Now fix
  $z \in F_{Z}$.  Since $Z$ is surjective,
  $\{\omega \in \Omega: Z(\omega)=z\}\neq \emptyset$.  Observing that
  $\displaystyle\cup_{y \in F_{Y}} \{\omega \in \Omega:
  Y(\omega)=y\}=\Omega$ we can find $y \in F_{Y}$ such that
  $\{\omega \in \Omega: Z(\omega)=z\}\cap \{\omega \in \Omega:
  Y(\omega)=y\}\neq \emptyset$.  But
  $\{\omega \in \Omega: Y(\omega)=y\}$
  $\in \sigma(Y)\subseteq \sigma(Z)$.  Therefore
  $\{\omega \in \Omega: Z(\omega)=z\} \subseteq \{\omega \in \Omega:
  Y(\omega)=y\}$ and this $y \in F_{Y}$ is unique.  For all such
  $z \in F_{Z}$ define $f(z)=y$.  We have therefore constructed
  $f: F_{Z} \rightarrow F_{Y}$ where $f(Z)=Y$.  To show that $f$ is
  $\cF_{Z}$-measurable, consider $E_{Y} \in \cF_{Y}$.  Then
  \begin{align*}
    f^{-1}(E_{Y})&=\{z \in F_{Z}: f(z)=y \in E_{Y}\} \\
                 &=\{z=Z(\omega) \in F_{Z}: f(Z(\omega))=Y(\omega)=y \in E_{Y}\} \\
                 &=Z(\{\omega \in \Omega: Y(\omega)=y \in E_{Y}\}).
  \end{align*}
  Since $E_{Y} \in \cF_{Y}$,
  $\{\omega \in \Omega: Y(\omega)=y \in E_{Y}\} \in \sigma(Y)
  \subseteq \sigma(Z)$.  Thus there exists $E_{Z} \in \cF_{Z}$ such
  that $Z(\{\omega \in \Omega: Y(\omega) \in E_{Y}\})=E_{Z}$ and we
  have proved that $f$ is $\cF_{Z}$-measurable.

  \noindent \itref{aa2} to \itref{aa1}: Let
  $A \in \sigma(Y)=\{Y^{-1}(B): B \in \cF_{Y}\}$.  Then there exists
  $B \in \cF_{Y}$ such that $Y^{-1}(B)=A$, which implies that
  $(f\circ Z)^{-1}(B)=A$ and thus $Z^{-1}(f^{-1}(B))=A$.  Since
  $B \in \cF_{Y}$ and $f$ is $\cF_{Z}$-measurable,
  $f^{-1}(B) \in \cF_{Z}$.  Also since $f^{-1}(B) \in \cF_{Z}$ and $Z$
  is $\sigma(Z)$-measurable, $Z^{-1}(f^{-1}(B)) \in \sigma(Z)$, which
  concludes the proof.
\end{proof0}


\begin{proof0}{\proporef{stocci}}\quad\\
  \noindent\itref{ascidef2} $\Rightarrow$ \itref{ascidef3}: We will
  prove \itref{scidef3} using the monotone class
  theorem~\cite[p.~235]{dur:13}.  Consider
  \begin{equation*}
    V=\left\{\text{$f$ real, bounded and measurable}: 
      \mE\left\{f(X) \mid Y,Z\right\}=\mE\left\{f(X) \mid Z\right\} \,\, \rm{a.s.}\, \right\}.
  \end{equation*}
  By linearity of expectation $V$ is a vector space of real and
  bounded functions.  Now by \itref{ascidef2}
  $\mathbbm{1}_{A_{X}} \in V$ for all $A_{X} \in \sigma(X)$.  Let
  $f_{n}(X) \in V$ for all $n \in \mN$ such that
  $0 \leq f_{n}(X) \uparrow f(X)$, where $f(X)$ is bounded.  Using
  conditional monotone convergence~\cite[p.~193]{dur:13},
  \begin{align*}
    \mE\left\{f(X) \mid Y,Z\right\}&= \lim_{n\rightarrow \infty} \mE\left\{f_{n}(X) \mid Y,Z\right\} \\
                                   &= \lim_{n\rightarrow \infty} \mE\left\{f_{n}(X) \mid Z\right\} \text{ since $f_{n}(X) \in V$}\\
                                   &= \mE\left\{f(X) \mid Z\right\}.
  \end{align*}
  Thus $f(X) \in V$ and we have shown \itref{scidef3}.

  \noindent\itref{ascidef3} $\Rightarrow$ \itref{ascidef4}: Let $f(X), g(Y)$ be
  real, bounded and measurable functions.  Then
  \begin{align*}
    \mE\left\{f(X)g(Y) \mid Z\right\}
    &= \mE\left[\mE\left\{f(X)g(Y) \mid Z,Y\right\} \mid Z\right] \,\, \rm{a.s.} \\
    &= \mE\left[g(Y) \mE\left\{f(X) \mid Z,Y\right\} \mid Z\right] \,\, \rm{a.s.}\\
    &= \mE\left[g(Y) \mE\left\{f(X) \mid Z\right\} \mid Z\right] \,\, \text{\rm{a.s.} by \itref{scidef3}}\\
    &= \mE\left\{f(X) \mid Z\right\} \mE\left\{g(Y) \mid Z\right\}
      \,\, \text{\rm{a.s.}}
  \end{align*}

  \noindent\itref{ascidef4} $\Rightarrow$ \itref{ascidef1}: Let $A_{X} \in
  \sigma(X)$ and $A_{Y} \in \sigma(Y)$.  Then $\mathbbm{1}_{A_{X}}$ is
  a real, bounded and $\sigma(X)$-measurable function and
  $\mathbbm{1}_{A_{Y}}$ is a real, bounded and $\sigma(Y)$-measurable
  function.  Thus \itref{ascidef1} is a special case of
  \itref{ascidef4}.
  
  \noindent\itref{ascidef1} $\Rightarrow$ \itref{ascidef2}: Let $A_{X} \in
  \sigma(X)$ and $w(Z)$ be a version of
  $\mE\left(\mathbbm{1}_{A_{X}} \mid Z\right)$.  Then $w(Z)$ is
  $\sigma(Z)$-measurable and so, since
  $\sigma(Z) \subseteq \sigma(Y,Z)$, also $\sigma(Y,Z)$-measurable.
  To show that
  $\mE\left(\mathbbm{1}_{A_{X}}\mathbbm{1}_{A}\right)=\mE\left\{w(Z)\mathbbm{1}_{A}\,\right\}$
  whenever $A \in \sigma(Y,Z)$, let
  \begin{equation*}
    \cD_{A_{X}}=\left\{A \in \sigma(Y,Z):\mE\left(\mathbbm{1}_{A_{X}}\mathbbm{1}_{A}\right)
      =\mE\left\{w(Z)\mathbbm{1}_{A}\right\}\,\right\}
  \end{equation*} 
  and
  \begin{equation*}
    \Pi=\{A \in \sigma(Y,Z): A=A_{X}\cap A_{Y} \text{ for some } A_{Y} \in \sigma(Y),A_{Z} \in \sigma(Z)\}.
  \end{equation*}
  Then $\sigma(\Pi)=\sigma(Y,Z)$~\cite[p.~73]{res:14}.  We will show
  that $\cD_{A_{X}}$ is a $d$-system that contains $\Pi$ and
  apply Dynkin's lemma~\cite[p.~42]{bil:95} to conclude that $\cD_{A_{X}}$ contains $\sigma(\Pi)=\sigma(Y,Z)$.  \\ \\
  To show that $\cD_{A_{X}}$ contains $\Pi$, let
  $A_{Y,Z}=A_{Y}\cap A_{Z}$ with $A_{Y} \in \sigma(Y)$ and
  $A_{Z} \in \sigma(Z)$.  Then
  \begin{align*}
    \mE\left(\mathbbm{1}_{A_{X}}\mathbbm{1}_{A}\right)
    &=\mE\{\mE\left(\mathbbm{1}_{A_{X}}\mathbbm{1}_{A_{Y}\cap A_{Z}} \mid Z\right)\} \\
    &=\mE\left\{\mathbbm{1}_{A_{Z}}\mE\left(\mathbbm{1}_{A_{X}}\mathbbm{1}_{A_{Y}} \mid Z\right)\right\} \\
    &=\mE\left\{\mathbbm{1}_{A_{Z}}\mE(\mathbbm{1}_{A_{X}} \mid
      Z)\mE(\mathbbm{1}_{A_{Y}} \mid Z)\right\}
      \,\, \text{by \itref{scidef1}} \\
    &=\mE\left[\mE\left\{\mE(\mathbbm{1}_{A_{X}} \mid Z)\mathbbm{1}_{A_{Y}\cap A_{Z}} \mid Z\right\}\right]  \\
    &=\mE\left\{\mE(\mathbbm{1}_{A_{X}}\mid Z)\mathbbm{1}_{A}\right\}.
  \end{align*}
  To show that $\cD_{A_{X}}$ is a $d$-system, first note that
  $\Omega \in \cD_{A_{X}}$.  Also, for $A_{1},A_{2} \in \cD_{A_{X}}$
  such that $A_{1} \subseteq A_{2}$,
  \begin{align*}
    \mE\left(\mathbbm{1}_{A_{X}}\mathbbm{1}_{A_{2}\setminus
    A_{1}}\right) &=
                    \mE\left(\mathbbm{1}_{A_{X}}\mathbbm{1}_{A_{2}}\right)
                    -\mE\left(\mathbbm{1}_{A_{X}}\mathbbm{1}_{A_{1}}\right) \\
                  &=
                    \mE\left\{w(Z)\mathbbm{1}_{A_{2}}\right\}-\mE\left\{w(Z)\mathbbm{1}_{A_{1}}\right\}
                    \,\, \text{since $A_{1},A_{2} \in \cD_{A_X}$} \\
                  &= \mE\left\{w(Z)\mathbbm{1}_{A_{2}\setminus A_{1}}\right\}.
  \end{align*}
  Now consider $(A_{n}: n \in \mN)$, an increasing sequence in
  $\cD_{A_{X}}$.  Then $A_{n} \uparrow \displaystyle\cup_{k}A_{k}$ and
  $\mathbbm{1}_{A_{X}}\mathbbm{1}_{A_{n}} \uparrow
  \mathbbm{1}_{A_{X}}\mathbbm{1}_{\displaystyle\cup_{k}A_{k}}$
  pointwise.  Thus by monotone convergence
  $\mE\left(\mathbbm{1}_{A_{X}}\mathbbm{1}_{A_{n}}\right) \uparrow
  \mE\left(\mathbbm{1}_{A_{X}}\mathbbm{1}_{\displaystyle\cup_{k}A_{k}}\right)$.
  Also
  $w(Z)\mathbbm{1}_{A_{n}} \uparrow
  w(Z)\mathbbm{1}_{\displaystyle\cup_{k}A_{k}}$ pointwise.  Thus by
  monotone convergence,
  $\mE\left\{w(Z)\mathbbm{1}_{A_{n}}\right\} \uparrow
  \mE\left\{w(Z)\mathbbm{1}_{\displaystyle\cup_{k}A_{k}}\right\}$.

  We can now see that:
  \begin{align*}
    \mE\left(\mathbbm{1}_{A_{X}}\mathbbm{1}_{\displaystyle\cup_{n}A_{n}}\right)
    &=\lim_{n}\mE\left(\mathbbm{1}_{A_{X}}\mathbbm{1}_{A_{n}}\right) \\
    &=\lim_{n}\mE\left\{w(Z)\mathbbm{1}_{A_{n}}\right\} \\
    &=\mE\left\{w(Z)\mathbbm{1}_{\displaystyle\cup_{n}A_{n}}\right\}.
  \end{align*}
\end{proof0}


\begin{proof0}{\proporef{vieq}}\quad\\
  \itref{vid1} $\Rightarrow$ \itref{vid2}: It readily follows from
  \itref{vid1} since $R(X \mid y,z)= R(X|z)$
  which is a function of $z$ only.\\
  \itref{vid2} $\Rightarrow$ \itref{vid1}: Let $a(Z):= R(X \mid Y,Z)$
  and let $(y,z) \in R(Y,Z)$.  Then
  \begin{align*}
    R(X \mid z)&= \{X(\sigma): \sigma \in \cS, Z(\sigma)=z \} \\
               &= \bigcup_{y' \in R(Y \mid z)} \{X(\sigma): \sigma \in \cS, (Y,Z)(\sigma)=(y',z) \} \\
               &= \bigcup_{y' \in R(Y \mid z)} R(X \mid y',z) \\
               &= \bigcup_{y' \in R(Y \mid z)} a(z)  \\
               &= a(z) \\
               &= R(X \mid y,z).
  \end{align*} 
  \itref{vid1} $\Rightarrow$ \itref{vid3}: Let $z \in R(Z)$.  Then
  \begin{align*}
    R(X,Y \mid z)&= \{(X,Y)(\sigma): \sigma \in \cS, Z(\sigma)=z \} \\
                 &= \bigcup_{y \in R(Y \mid z)} \bigcup_{x \in R(X \mid y,z)} \{(x,y)\} \\
                 &= \bigcup_{y \in R(Y \mid z)} \bigcup_{x \in R(X \mid z)} \{(x,y)\} \quad \text{by \itref{vid1}}\\
                 &= R(X \mid z) \times R(Y \mid z).
  \end{align*}
  \itref{vid3} $\Rightarrow$ \itref{vid1}: Let $(y,z) \in R(Y,Z)$.
  Then
  \begin{align*}
    R(X \mid y,z)&=\{x : x \in R(X \mid y,z)\}\\
                 &= \{x : (x,y) \in R(X,Y \mid z)\} \\
                 &= \{x : (x,y) \in R(X \mid z) \times R(Y \mid z)\} \quad \text{by \itref{vid3}} \\
                 &= \{x : x \in R(X \mid z), y \in R(Y \mid z)\} \\
                 &= \{x : x \in R(X \mid z)\} \\
                 &= R(X \mid z).
  \end{align*}
\end{proof0}

\begin{proof0}{\proporef{eqlemma}}\quad\\
  \itref{avim1} to \itref{avim2}: For $y \in Y(\cS)$, take
  $\sigma_{y}\in Y^{-1}(\{y\})$.  Define
  $f: Y(\cS) \rightarrow W(\cS)$ by $f(y)=W(\sigma_{y})$.  To check
  that $f(Y)=W$, let $\sigma \in \cS$ and $y=Y(\sigma)$.  Then
  $f(Y(\sigma))=f(y)=W(\sigma_{y})$.  By \itref{vim1} (since
  $Y(\sigma)=Y(\sigma_{y})=y$), we get that $W(\sigma_{y})=W(\sigma)$.
  Thus $f(Y(\sigma))=W(\sigma)$ for all $\sigma \in \cS$.

  \noindent\itref{avim2} to \itref{avim1}: Follows from the definition of
  $\preceq$.
\end{proof0}

\begin{proof0}{\proporef{nsprop1}}\quad\\ 
  \itref{nspr1} $\Rightarrow$ \itref{nspr2}.  Let $\phi \in \Phi(\cS)$
  and let $h_{1}$ be a real, bounded and measurable function.  Then
  for $\sigma \in \Phi^{-1}(\phi)$ and $h_{2}$ a real bounded and
  measurable function, we have:
  \begin{align*}
    \mE_\sigma\left\{h_{1}(X)h_{2}(Y)\mid Z\right\}&=
                                                     \mE_\sigma\left[\mE_\sigma\left\{h_{1}(X)h_{2}(Y)\mid
                                                     Y,Z\right\}\mid Z\right]
                                                     \,\, \text{\as{\mP_\sigma}}\\
                                                   &= \mE_\sigma\left[h_{2}(Y) \mE_\sigma\left\{h_{1}(X)\mid
                                                     Y,Z\right\}\mid Z\right]
                                                     \,\, \text{\as{\mP_\sigma}}\\
                                                   &= \mE_\sigma\{h_{2}(Y) w_{\phi}(Z)\mid Z\} \,\, \text{\as{\mP_\sigma} by (i) }\\
                                                   &= w_{\phi}(Z)\, \mE_\sigma\left\{h_{2}(Y)\mid Z\right\} \,\,
                                                     \text{\as{\mP_\sigma}}.
  \end{align*}

  \noindent \itref{nspr2}$\Rightarrow$ \itref{nspr3}.  Follows
  directly on taking $h_1(X) = \mathbbm{1}_{A_{X}}$,
  $h_2(Y) = \mathbbm{1}_{A_{Y}}$.

  \noindent \itref{nspr3}$\Rightarrow$ \itref{nspr4}.  Let
  $\phi \in \Phi(\cS)$ and $A_{X} \in \sigma(X)$ and consider
  $w_{\phi}(Z)$ as in \itref{nspr3}.  Note that~\eqref{nstoceq} is
  equivalent to: for all $\sigma \in \Phi^{-1}(\phi)$,
  \begin{equation}
    \label{eq:1}
    \mE_\sigma\left(\mathbbm{1}_{A_{X}}\mathbbm{1}_{A_{Y,Z}}\right)=
    \mE_\sigma\left\{w_{\phi}(Z)\mathbbm{1}_{A_{Y,Z}}\right\} 
    \,\, \text{whenever $A_{Y,Z} \in \sigma(Y,Z)$}.
  \end{equation}
  To show~\eqref{1}, consider
  \begin{equation*}
    \cD_{A_{X}}=\left\{A_{Y,Z} \in \sigma(Y,Z):\mE_\sigma\left(\mathbbm{1}_{A_{X}}\mathbbm{1}_{A_{Y,Z}}\right)
      =\mE_\sigma\{w_{\phi}(Z)\mathbbm{1}_{A_{Y,Z}}\}\,\right\}
  \end{equation*} 
  and
  \begin{equation*}
    \Pi=\{A_{Y,Z} \in \sigma(Y,Z): A_{Y,Z}=A_{X}\cap A_{Y} \text{ for some } A_{Y} \in \sigma(Y),A_{Z} \in \sigma(Z)\}.
  \end{equation*}
  Then $\sigma(\Pi)=\sigma(Y,Z)$~\cite[p.~73]{res:14}.  We will show
  that $\cD_{A_{X}}$ is a $d$-system that contains $\Pi$; it will then
  follow from Dynkin's lemma~\cite[p.~42]{bil:95} that $\cD_{A_{X}}$
  contains $\sigma(\Pi)=\sigma(Y,Z)$
  and hence that~\eqref{1} holds.  \\ \\
  To show that $\cD_{A_{X}}$ contains $\Pi$, take
  $A_{Y} \in \sigma(Y)$ and $A_{Z} \in \sigma(Z)$, and let
  $A_{Y,Z}=A_{Y}\cap A_{Z}$.  Then
  \begin{align*}
    \mE_\sigma\left(\mathbbm{1}_{A_{X}}\mathbbm{1}_{A_{Y,Z}}\right)
    &=\mE_\sigma\left\{\mE_\sigma\left(\mathbbm{1}_{A_{X}}\mathbbm{1}_{A_{Y}\cap A_{Z}}\mid Z\right)\right\}\\
    &=\mE_\sigma\left\{\mathbbm{1}_{A_{Z}}\mE_\sigma\left(\mathbbm{1}_{A_{X}}\mathbbm{1}_{A_{Y}}\mid Z\right)\right\}\\
    &=\mE_\sigma\left\{\mathbbm{1}_{A_{Z}}w_{\phi}(Z)\mE_\sigma(\mathbbm{1}_{A_{Y}}\mid
      Z)\right\}
      \,\, \text{by (iii)} \\
    &=\mE_\sigma\left[\mE_\sigma\left\{w_{\phi}(Z)\mathbbm{1}_{A_{Y}\cap A_{Z}}\mid Z\right\}\right] \\
    &=\mE_\sigma\left\{w_{\phi}(Z)\mathbbm{1}_{A_{Y,Z}}\right\}.
  \end{align*}

  To show that $\cD_{A_{X}}$ is a $d$-system, first note that
  $\Omega \in \cD_{A_{X}}$.  Also, for $A_{1},A_{2} \in \cD_{A_{X}}$
  such that $A_{1} \subseteq A_{2}$, we have:
  \begin{align*}
    \mE_\sigma\left(\mathbbm{1}_{A_{X}}\mathbbm{1}_{A_{2}\setminus
    A_{1}}\right) &=
                    \mE_\sigma\left(\mathbbm{1}_{A_{X}}\mathbbm{1}_{A_{2}}\right)
                    -\mE_\sigma\left(\mathbbm{1}_{A_{X}}\mathbbm{1}_{A_{1}}\right) \\
                  &=
                    \mE_\sigma\left\{w_{\phi}(Z)\mathbbm{1}_{A_{2}}\right\}-\mE_\sigma\left\{w_{\phi}(Z)\mathbbm{1}_{A_{1}}\right\}
                    \,\, \text{since $A_{1},A_{2} \in \cD_{A_X}$} \\
                  &= \mE_\sigma\left\{w_{\phi}(Z)\mathbbm{1}_{A_{2}\setminus
                    A_{1}}\right\}.
  \end{align*}
  Now consider an increasing sequence $(A_{n}: n \in \mN)$ in
  $\cD_{A_{X}}$.  Then $A_{n} \uparrow \displaystyle\cup_{k}A_{k}$ and
  $\mathbbm{1}_{A_{X}}\mathbbm{1}_{A_{n}} \uparrow
  \mathbbm{1}_{A_{X}}\mathbbm{1}_{\displaystyle\cup_{k}A_{k}}$
  pointwise.  Thus by monotone convergence
  $\mE_\sigma(\mathbbm{1}_{A_{X}}\mathbbm{1}_{A_{n}}) \uparrow
  \mE_\sigma(\mathbbm{1}_{A_{X}}\mathbbm{1}_{\displaystyle\cup_{k}A_{k}})$.
  Also
  $w_{\phi}(Z)\mathbbm{1}_{A_{n}} \uparrow
  w_{\phi}(Z)\mathbbm{1}_{\displaystyle\cup_{k}A_{k}}$ pointwise, and
  thus by monotone convergence
  $\mE_\sigma\{w_{\phi}(Z)\mathbbm{1}_{A_{n}}\} \uparrow
  \mE_\sigma\{w_{\phi}(Z)\mathbbm{1}_{\displaystyle\cup_{k}A_{k}}\}$.
  We can now see that
  \begin{align*}
    \mE_\sigma\left(\mathbbm{1}_{A_{X}}\mathbbm{1}_{\displaystyle\cup_{n}A_{n}}\right)
    &=\lim_{n}\mE_\sigma\left(\mathbbm{1}_{A_{X}}\mathbbm{1}_{A_{n}}\right) \\
    &=\lim_{n}\mE_\sigma\left\{w_{\phi}(Z)\mathbbm{1}_{A_{n}}\right\}
      \,\, \text{since $A_{n} \in \cD_{A_X}$}\\
    &=\mE_\sigma\left\{w_{\phi}(Z)\mathbbm{1}_{\displaystyle\cup_{n}A_{n}}\right\}.
  \end{align*}

  \noindent \itref{nspr4}$\Rightarrow$ \itref{nspr1}.  We will prove
  \itref{nspr4} using the monotone class
  theorem~\cite[p.~235]{dur:13}.
  Let $\phi \in \Phi(\cS)$ and consider \\ \\*
  $V:=\{h \text{ real, bounded and measurable: there exists } w_{\phi}(Z) \text{ such that, for all } \sigma \in \Phi^{-1}(\phi)$, \\
  $\mE_\sigma\{h(X)\mid Y,Z\}=w_{\phi}(Z) \,\, \as{\mP_\sigma}\}.$  \\ \\*
  By linearity of expectation, $V$ is a vector space of real and
  bounded functions, and by \itref{nspr4} $\mathbbm{1}_{A_{X}} \in V$
  for all $A_{X} \in \sigma(X)$.  Now let $h_{n} \in V$ for all
  $n \in \mathbb{N}$ such that $0 \leq h_{n} \uparrow h$ with $h$
  bounded.  Using conditional monotone
  convergence~\cite[p.~193]{dur:13},
  \begin{align*}
    \mE_\sigma\left\{h(X)\mid Y,Z\right\}
    &= \lim_{n\rightarrow \infty} \mE_\sigma\left\{h_{n}(X)\mid Y,Z\right\} \\
    &= \lim_{n\rightarrow \infty} w_{\phi}^{n}(Z)\,\, \text{since $h_{n} \in V$}\\
    &=: w_{\phi}(Z)
  \end{align*}
  which proves that $h \in V$.  By the monotone class
  theorem~\cite[p.~235]{dur:13}, $V$ contains every bounded measurable
  function, and thus we have shown \itref{nspr1}.
\end{proof0}


\bibliography{cibib} \bibliographystyle{plainnat}
\end{document}